\documentclass[
  a4paper, 
  reqno, 
  oneside, 
  11pt
]{amsart}

\usepackage[dvipsnames]{xcolor} % Must come before tikz!
\colorlet{cite}{LimeGreen!50!Green}
\usepackage{tikz}  % for diagrams
\usetikzlibrary{arrows,positioning}
\tikzset{ 
  baseline=-2.3pt,
  text height=1.5ex, text depth=0.25ex,
  >=stealth,
  node distance=2cm,
  mid/.style={fill=white,inner sep=2.5pt},
}

\usepackage{lmodern}
\usepackage{tikz-cd}
\usepackage{amsthm, amssymb, amsfonts}

\usepackage{graphicx,caption,subcaption}
\usepackage{braket}  % For nice sets
\usepackage{microtype}

\usepackage[%
  bookmarks=true,			% show bookmarks bar?
  unicode=true,			% non-Latin characters in Acrobat’s bookmarks
  pdftitle={Deformations of noncompact manifolds},		%
  pdfauthor={Elizabeth Gasparim}{Francisco Rubilar},	% 
  pdfkeywords={Hodge structures}{adjoint orbit}{compactification}{Lefschetz fibration},	% list of keywords
  colorlinks=true,		% false: boxed links; true: colored links
  linkcolor=Blue,			% color of internal links
  citecolor=cite,		% color of links to bibliography
  filecolor=magenta,		% color of file links
  urlcolor=RoyalBlue			% color of external links
]{hyperref}				% One of the most useful packages I have found.
\usepackage{cleveref}
\newtheoremstyle{mydef}
  {}		% Space above environment
  {}		% Space below environment
  {}		% Body font
  {}		% Indent amount (empty = no indent, \parindent = para indent)
  {\scshape}	% theorem head font
  {. }		% Punctuation after heading
  { }		% Space after heading
  {\thmname{#1}\thmnumber{ #2}\thmnote{ #3}}	% Heading spec

\theoremstyle{plain}	% 'plain' is the default.  The others are 'definition' and 'remark'.
\newtheorem{theorem}{Theorem} % putting [section] on the end here tells latex to number the theorem environment within sections, ie. Theorem 2.3 for the third theorem in section 2.
\newtheorem{lemma}[theorem]{Lemma} % putting [theorem] in the middle here tells latex to use the same counter as the theorem environment.  See the output to see how this works.
\newtheorem{proposition}[theorem]{Proposition}
\newtheorem*{theorem*}{Theorem}
\newtheorem{corollary}[theorem]{Corollary}

\theoremstyle{mydef} % Here we have used the custom theoremstyle defined above instead of the usual 'definition' style.
\newtheorem{definition}[theorem]{Definition}
\newtheorem*{conjecture*}{Conjecture}
\theoremstyle{remark}
\newtheorem{remark}[theorem]{Remark}
\newtheorem{notation}[theorem]{Notation}

\DeclareMathOperator{\Ad}{Ad}
\DeclareMathOperator{\grad}{grad}
\DeclareMathOperator{\ham }{ham}

\newtheorem*{proposition*}{Proposition}
\newtheorem*{lemma*}{Lemma}
\newtheorem*{corollary*}{Corollary}
\theoremstyle{definition}
\theoremstyle{remark}

\newcommand{\ce}{\mathrel{\mathop:}=}

\author{Elizabeth Gasparim {\tiny and}  Luiz A. B. San Martin}
\address{E. G. - Depto. Matem\'aticas, Univ. Cat\'olica del Norte, Antofagasta, Chile. \newline
L.  SM. - Imecc -
Unicamp, Depto. de Matem\'{a}tica, Campinas, Brasil.\newline
 etgasparim@gmail.com,  smartin@ime.unicamp.br}
\title{Morse functions and Real Lagrangian Thimbles on Adjoint Orbits}

\begin{document}
\maketitle

\begin{abstract}
We compare Lagrangian thimbles  for the potential of a  Landau--Ginzburg model  to the 
Morse theory of its real part. We  explore 
Landau--Ginzburg models defined using Lie theory, constructing their  real Lagrangian thimbles
explicitly and comparing them to the stable and unstable manifolds of the real gradient flow.
 \end{abstract}

\tableofcontents

\section{Real and complex Morse functions}

Given a  real manifold $M$, a smooth function  $f\colon M \rightarrow \mathbb R$ is called a {\it Morse function} if it 
has only nondegenerate critical points. Recall that a critical point $p$ of $f$ is nondegenerate if the Hessian matrix
$\frac{\partial^2f}{\partial x_i\partial x_j}(p)$ is nonsingular. 
Nondegenerate critical points are isolated, and  the lemma of Morse tells us that, on 
a neighborhood of such a critical point, the Morse function can be written in local coordinates as 
$$f=f(p)-x_1^2-\cdots -x_\lambda^2+x_{\lambda+1}^2+ \cdots +x_n^2 .$$
The integer $\lambda$ is called the index of $f$ at $p$.
Morse theory  tells us how to recover the topology of a compact manifold $M$ from the data of  indices of  critical points of $f$. 
In fact, $M$ has the homotopy type of a finite CW-complex with one cell of dimension $\lambda$ for each critical point of index $\lambda$, 
see \cite{M}.

Given a complex  manifold $M$, a complex function $f\colon M \rightarrow \mathbb C$ is called a
 {\it Landau--Ginzburg model} and the function $f$ is called the {\it superpotential}. If in addition $df$ is surjective outside a finite number of points 
 and $f$ is a holomorphic Morse function, that is, it has only 
 nondenegerate critical points, then $f$ is called a {\it Topological Lefschetz fibration}. 
On a neighborhood of  a critical point, a Lefschetz fibration  can be written in local coordinates as 
$$f=f(p)+z_1^2+ \cdots +z_n^2 .$$

Note that changing the coordinate $z_j$ to $iz_j$ takes $z_j^2$ to $-z_j^2$ so that 
it is meaningless to talk about indices of critical points in the complex case. 
Furthermore, observe that if the complex manifold $M$ is compact, then any holomorphic function $f\colon M \rightarrow \mathbb C$ 
is constant. Hence, this type of Landau--Ginzburg model is interesting only in the 
case when $M$ is noncompact. For compact manifolds the natural concept of Landau--Ginzburg model  is
$f\colon M \rightarrow \mathbb P^1$. 

In algebraic geometry functions $f\colon M \rightarrow \mathbb P^1$ give rise to so-called pencils, 
which are studied extensively within the context of 
Picard--Lefschetz theory. Fibers of such pencils intersect in the base locus, and 
a topological Lefschetz fibration can be obtained  by blowing up this base locus. 
If $M_b$ is a regular fibre contained in a small neighborhood of a singular fibre $M_o$
then there is a retraction $M_b \rightarrow M_o$ which induces a surjection in homology
$H_*(M_b) \rightarrow H_*(M_o)$. The classes in the kernel of this surjection are called
{\it vanishing cycles} and  are 
important objects of study in Hodge theory, see \cite{PS}. The fundamental theorem of Picard--Lefschetz theory describes the
intersection theory of vanishing cycles in the case when $M$ is a projective variety.
 In particular, for  compact $M$ this fundamental theorem implies that  each critical point of $f$ has a 
corresponding  vanishing cycle. However, in the noncompact case  existence of a vanishing cycle 
corresponding to each critical point is not guaranteed, see \cite{S}.

If $(M, \omega)$ is a symplectic manifold, then a topological Lefschetz fibration $f\colon M \rightarrow \mathbb C$
 is called a {\it Symplectic Lefschetz fibration} provided
the symplectic form  $\omega$ is nondegenerate on the fibre $M_x$
for all $x$ in the sense that:
\begin{itemize}
\item[$\iota.$]   $M_{x}$ is a symplectic submanifold of  $M$ for each regular value $x$, and 
\item[$\iota\iota.$]  for each  critical point $p$ the symplectic form $\omega_{p}$
is non degenerate over the tangent cone of
$M_{f(p)}$ at $p$.
\end{itemize}

For any symplectic fibration there exists a natural connection obtained by taking the symplectic orthogonal 
to the fibre. If $o$ is a critical value of $f$ and $b$ is a regular value contained in a neighborhood of $o$, 
then consider a path $\lambda\colon [0,1] \rightarrow \mathbb C $  from $\lambda(0)=b$ to $\lambda(1)=o$. 
Given a vanishing cycle $\alpha \subset M_b$ we can 
use the connection to  parallel transport the cycle $\alpha$ along $\lambda$ 
all the way to the corresponding critical point $p$. For each $t \in [0,1]$ we obtain a cycle $\alpha_t \subset M_{\lambda(t)}$ 
so that $\alpha_0 = \alpha$ and $\alpha_1 = p$. The object traced by the cycle $\alpha$ on its way to $p$ is topologically a 
closed disc $D = \{\cup_{t\in [0,1]}\alpha_t\}$ with boundary $\partial D = \alpha$  and is called a {\it thimble}.

Vanishing cycles live naturally in the middle homology of the regular fibre, hence $\dim_{\mathbb R} \alpha = \dim_{\mathbb C} M_b$.
Thus, it makes sense ask whether $\alpha$ is a Lagrangian submanifold of $M_b$, that is, if $\omega$ vanishes on $\alpha$, 
and in the affirmative case
$\alpha$ is  called a {\it Lagrangian vanishing cycle}. If the corresponding thimble is  a Lagrangian submanifold 
of $M$ it is then called a {\it Lagrangian thimble}. Lagrangian thimbles are the objects that generate the so-called 
{\it Fukaya--Seidel category} of the fibration, and they are our main objects of study in this paper. 

We explore symplectic Lefschetz fibrations on semisimple adjoint orbits, recalling the construction 
of the complex superpotential  in section \ref{SLF} and describing the gradient vector field of 
its real part  in section \ref{RGR}. We then construct Lagrangian vanishing cycles in section \ref{LVC} 
and Lagrangian thimbles of a preferred type which we name {\it real Lagrangian thimbles} (definition \ref{rlt}) 
obtained using the Morse theory of the real part of the superpotential.

Profiting from  the knowledge  of  Lagrangian submanifolds of the adjoint orbits described in \cite{GGSM2} 
and existence of Lagrangian submanifolds inside their compactifications described in \cite{GSMV} we 
have existence of  Lagrangian submanifolds $V$ passing through any critical value $c$
of the superpotential $f_H= f_1+if_2$ and containing a real sphere 
that is a vanishing cycle for $f_1$ constructed in section \ref{LVC}.
Then, considering the restriction of the real part $g_1= f_1\vert_V$ to the Lagrangian submanifold $V$ of $\mathcal O(H_0)$
we are able to find out explicitly the desired real thimbles:

\begin{theorem*}[\ref{stable/unstable}]
	Take $c$ 
	near $f_{1}\left( x\right) =g_{1}\left( x\right) $. We have that
	\begin{equation*}
	g_{1}^{-1}\left[ c,g_{1}\left( x\right) \right] =f_{1}^{-1}\left[
	c,f_{1}\left( x\right) \right] \cap V \quad \text{in the negative definite case, or } 
	\end{equation*}%
	  \begin{equation*}g_{1}^{-1}\left[ g_{1}\left( x\right) ,c%
	\right] =f_{1}^{-1}\left[ f_{1}\left( x\right) ,c\right] \cap V \quad \text{in the 
	positive definite case}\phantom{xxx} \end{equation*}
	 is homeomorphic to a closed ball in  $\mathbb{R}^{\dim
		V}.$ This ball is a  Lagrangian thimble.
\end{theorem*}

 We provide examples that  illustrate the behaviour of the superpotential 
  over Lagrangian submanifolds obtained from graphs in section \ref{PG}. 
  Finally, exploring the graph of $\Gamma(R_{w_{0}})$ of the right translation by the principal involution of the Weyl group,
  we explicitly describe examples of the relation between the Morse theory of the real part 
and the real Lagrangian thimbles of the superpotential, concluding this work with:

\begin{theorem*}[\ref{Morse}]
	%Let $m_{j}^{\pm }$ be as in definition \ref{defememaismenos}. Then,
	The stable and unstable manifolds of  $\grad \left( 
	\text{Re}f_{H}\right) $ at the critical point  $\left[ e_{j}\right] $ are
	open in  the graph $\Gamma( m_{j}^{\pm }\circ R_{w_{0}}) $. The 
	real  Lagrangian thimbles are closed balls contained in  the graph $\Gamma
	( m_{j}^{\pm }\circ R_{w_{0}}) $.
\end{theorem*}
 
 The relation between real thimbles in symplectic Lefschetz  fibrations and stable and unstable manifolds 
 of the gradient flow is part of the folklore of the subject and is presumably well known to experts. Nevertheless, 
 we were unable to find such relation explained  in detail anywhere
 in the literature, and we believe that the explicit constructions 
 given here are useful illustrations of the construction of Lagrangians.

\section{Symplectic Lefschetz fibrations}\label{SLF}

In this section we summarize the construction of symplectic Lefschetz fibrations on adjoint orbits, 
their compactifications and Lagrangian submanifolds discussed in \cite{GGSM1,GGSM2,BGGSM}.

Let $G$ be a complex semisimple Lie group with Lie algebra $\mathfrak{g}$ and denote by $\langle X,Y\rangle:=\textnormal{tr}(\textnormal{ad}(X),\textnormal{ad}(Y))$ the Cartan--Killing form of $\mathfrak{g}$. Fix a Cartan subalgebra $\mathfrak{h}\subset \mathfrak{g}$ and a real compact form $\mathfrak{u}$ of $\mathfrak{g}$. Associated to these subalgebras are the subgroups $T=\langle \exp \mathfrak{h}\rangle=\exp \mathfrak{h}$  and $U=\langle \exp \mathfrak{u}\rangle=\exp \mathfrak{u}$. Denote by $\tau$ the conjugation associated to $\mathfrak{u}$ which is defined by $\tau(X)=X$ if $X\in\mathfrak{u}$ and $\tau(Y)=-Y$ if $Y\in i\mathfrak{u}$, that is, if $Z=X+iY\in \mathfrak{g}$ with $X,Y\in \mathfrak{u}$ then $\tau(X+iY)=X-iY$. In this case we can define the Hermitian form $\mathcal{H}_{\tau }\colon \mathfrak{g}\times \mathfrak{g}\rightarrow \mathbb{C}$ as
\begin{equation}
\mathcal{H}_{\tau}(X,Y):=-\langle X,\tau Y\rangle.
\end{equation}
If we write the real and imaginary parts of $\mathcal{H}_{\tau}$ as
$$\mathcal{H}_{\tau}(X,Y)=(X,Y)+i\Omega(X,Y),$$
it is well known that the real part $(\cdot, \cdot)$ is an inner product and the imaginary part $\Omega$ is a symplectic form on $\mathfrak{g}$. Indeed, we have
\begin{equation*}
0\neq i\mathcal{H}\left( X,X\right) =\mathcal{H}\left( iX,X\right) =i\Omega
\left( iX,X\right) , 
\end{equation*}%
that is, $\Omega \left( iX,X\right) \neq 0$ for all $X\in \mathfrak{g}$,
which shows that $\Omega $ is nondegenerate. Moreover, $d\Omega =0$ because $%
\Omega $ is a constant bilinear form. The fact that $\Omega \left( iX,X\right) \neq 0$ for all $X\in \mathfrak{g} $
guarantees that the restriction of $\Omega $ to any complex subspace of $%
\mathfrak{g}$ is also nondegenerate.\\

We denote by $\mathcal{O}\left( H_{0}\right)$ the adjoint orbit $\Ad G (H_0)$ of $H_0\in \mathfrak{g}$.
% Recall that for every $x\in\mathcal{O}( H_{0})$ we have $T_x \left( H_{0}\right)=\lbrace [A,x]:\ A\in \mathfrak{g}\rbrace$. Thus, as $\mathfrak{g}$ is a complex Lie algebra, we get that the tangent spaces to $\mathcal{O}( H_{0})$ are complex subspaces of $\mathfrak{g}$ since if $[A,x]$ is a tangent vector to $\mathcal{O}( H_{0})$ at $x$ then $i[A,x]=[iA,x]$ is also a tangent vector at $x$. This implies that every adjoint orbit is a complex submanifold.\\
Denote by $\mathfrak{h}^\ast$ the dual vector space of $\mathfrak{h}$ and by $\Pi$ the set of all roots associated to the Cartan subalgebra $\mathfrak{h}$. An element $H\in \mathfrak{h}$ is called regular if $\alpha(H)\neq 0$ for all $\alpha\in \Pi$. As the restriction of Cartan--Killing form to $\mathfrak{h}$ is nondegenerate, the map $\varphi\colon \mathfrak{h}\to \mathfrak{h}^\ast$ defined by $\varphi(X)=\langle X,\cdot \rangle$ is a linear isomorphism. We denote by $\mathfrak{h}_\mathbb{R}$  the real subspace of $\mathfrak{h}$ generated by  $\varphi^{-1}(\Pi)$.
The pullback of the symplectic form $\Omega$ by
the inclusion $\mathcal{O}\left( H_{0}\right) \hookrightarrow \mathfrak{g}$
defines a symplectic form on $\mathcal{O}\left( H_{0}\right) $.
With this choice of symplectic form, we have our construction of Symplectic Lefschetz fibrations via Lie theory as follows:
\begin{theorem}\label{lefsch} \cite[Thm. 2.2]{GGSM1} 
Let $\mathfrak{h}$ be the
	Cartan subalgebra of a complex semisimple Lie algebra $\mathfrak{g}$. Given $H_{0}\in 
	\mathfrak{h}$ and $H\in \mathfrak{h}_{\mathbb{R}}$ with $H$ a regular
	element. The \textit{height function} \ $f_{H}\colon\mathcal{O}\left(
	H_{0}\right) \rightarrow \mathbb{C}$ defined by 
	\begin{equation*}
	f_{H}\left( x\right) =\langle H,x\rangle \qquad x\in \mathcal{O}\left(
	H_{0}\right) 
	\end{equation*}%
	has a finite number (= $|\mathcal{W}|/|\mathcal{W}_{H_{0}}|$) of isolated
	singularities and gives $\mathcal{O}\left(H_{0}\right) $ the structure of a
	symplectic Lefschetz fibration. 
\end{theorem}
Here $\mathcal{W}=\textnormal{Nor}_G(\mathfrak{h})/\textnormal{Cen}_G(\mathfrak{h})$ denotes the Weyl group.
In the language used in Homological Mirror Symmetry the pair $(\mathcal O(H_0), f_H)$ is called a
{ Landau--Ginzburg model} with {superpotential}  $f_H$. See \cite{BBGGSM} for a discussion of 
the mirror of $\mathcal O(H_0)$ in the case of $\mathfrak{sl}(2,\mathbb C)$. 

Given a regular element $H_0 \in \mathfrak g$, consider the set $\Theta$ of 
simple roots that have $H_0$ in their kernel. Let  $\mathfrak{p}_{\Theta }$ be 
the  parabolic subalgebra determined by $\Theta$, with corresponding parabolic  subgroup $P_{\Theta }$.
The quotient  $F_\Theta :=G/P_\Theta$ by the parabolic subgroup  is the flag manifold 
determined by $H_0$. Another regular element in $\mathfrak g$ will correspond to the same 
flag manifold if it is annihilated by the same set of roots $\Theta$, so for questions 
regarding the isomorphism with $T^* F_\Theta$ we denote the regular element  by 
$H_\Theta$ instead of $H_0$.\\

The adjoint orbit of a regular element $H_\Theta$ is
isomorphic to the cotangent bundle of the flag manifold $F_\Theta$  \cite[Thm.\thinspace 2.1]{GGSM2}. The
isomorphism $\iota\colon \mathcal{O}\left( H_{\Theta}\right)
\rightarrow T^{\ast }\mathbb{F}_{\Theta }$ is obtained observing that 
$$\mathcal{O}\left(
H_{\Theta }\right) =\bigcup_{k\in K}\mathrm{Ad}\left( k\right) \left(
H_{\Theta }+\mathfrak{n}_{\Theta }^{+}\right), $$ then taking for each $X\in 
\mathfrak{n}_{\Theta }^{+}$, the correspondence: 
\begin{equation*}
\mathrm{Ad}\left( k\right) \left( H_{\Theta }+X\right) \mapsto \langle 
\mathrm{Ad}\left( k\right) X,\cdot\rangle
\end{equation*}
where $\mathrm{Ad}\left( k\right) \mathfrak{n}_{\Theta }^{-}$ is identified
with the tangent space $T_{kb_{\Theta }}\mathbb{F}_{\Theta }$, where 
$b_\Theta$ is the origin of the flag.

Let $\mu $ be the moment map of the action $a\colon G\times T^{\ast }\mathbb{%
	F}_{\Theta }\rightarrow T^{\ast }\mathbb{F}_{\Theta }$. Then $\mu
\colon T^{\ast }\mathbb{F}_{\Theta }\rightarrow \mathrm{Ad}\left ( G\right)
H_{\Theta }$ is the inverse of the map $\iota \colon \mathrm{Ad}\left(
G\right) H_{\Theta }\rightarrow T^{\ast }\mathbb{F}_{\Theta }$, and satisfies 
\begin{equation*}
\mu ^{\ast }\omega =\Omega,
\end{equation*}
where $\Omega $ is the canonical symplectic form of $T^{\ast }\mathbb{F}%
_{\Theta }$ and $\omega $ the (real) Kirillov--Kostant--Souriau form on $%
\mathrm{Ad}\left( G\right) H_{\Theta }$.

We  compactify the total space of $T^{\ast }\mathbb{F}_{\Theta }$ to the
trivial product $F_\Theta\times F_{\Theta^\ast}$ as: 
\begin{equation*}
\mathcal{O}\left( H_{\Theta}\right) \overset{\sim}{ \rightarrow} T^{\ast }%
\mathbb{F}_{\Theta } \hookrightarrow \overline{T^{\ast }\mathbb{F}_{\Theta }}%
= F_\Theta\times F_{\Theta^\ast}.
\end{equation*}

\cite[Thm.\thinspace 5.3]{BGGSM} showed how to extend the potential $f_H$ 
to the compactification in the case of minimal adjoint orbits. 

Let $w_{0}$ be the principal involution of the Weyl group $\mathcal{W}$,
that is, the element of highest length as a product of simple roots. 
The right action $R_{w_{0}}\colon \mathbb{F}_{H_{0}}\rightarrow 
\mathbb{F}_{H_{0}^{\ast }}$ is anti-symplectic with
respect to the K\"{a}hler forms on $\mathbb{F}_{H_{0}}$ and $\mathbb{F}%
_{H_{0}^{\ast }}$ given by the Borel metric and canonical complex
structures. We use graphs of anti-symplectic maps to 
construct Lagrangian submanifolds of 
$\mathbb{F}_{H_{0}}$ and $\mathbb{F}%
_{H_{0}^{\ast }}$, these graphs will be used in the construction of  Lagrangian thimbles of $\mathcal O(H_0)$
in section \ref{RLT}.

\begin{notation}
	We will denote by $\Gamma(f)$ the graph of a map $f$.
\end{notation}

\begin{remark}\label{Rwo}
$\Gamma(R_{w_{0}})$, that is, the graph of $R_{w_{0}}$ the right translation by the principal involution of $\mathcal W$,
 is the orbit of $K$ by the diagonal
action. This orbit is the zero section of $T^{\ast }\mathbb{F}_{H_{0}}$
under the identification with $\mathcal{O}\left( H_{0}\right) \approx G\cdot
\left( H_{0},-H_{0}\right) $. Therefore, $\Gamma\left(
R_{w_{0}}\right) $ is a real Lagrangian submanifold of the product.
\end{remark}

%Further examples of real Lagrangian graphs by composites (either
%on the left or on the right) of $R_{w_{0}}$ with symplectic maps
%can then be obtained:
%
%\begin{theorem*}\cite[Thm. 6.3]{GGSM2}
%	For $k_1,k_2 \in K$ and for $m \in T$:
%	
%	\begin{enumerate}
%		\item $\Gamma\left(k_1\circ R_{w_{0}}\circ k_{2}\right)$ corresponds
%		to a Lagrangian submanifold of $\mathcal{O}\left( H_{\Theta }\right)$, and
%		
%		\item $\Gamma\left(m\circ R_{w_{0}}\right)$ corresponds to a
%		Lagrangian submanifold of $\mathcal{O}\left( H_{\Theta }\right).$
%	\end{enumerate}
%\end{theorem*}

%................................................................................................

\section{The gradient  field of  of $\mathrm{Re}f_H$}\label{RGR}

The field $Z\left( x\right) =[x,[\tau x,H]]$ is defined over the whole 
algebra $\mathfrak{g}$ and is tangent to the adjoint orbits, since the 
tangent space to  $\mathrm{Ad}\left( G\right) x$ at $x$ is the  image
of $\mathrm{ad}\left( x\right) $.
Assume here that both  $H$ and $H_{0}$ are regular and belong to the 
Weyl  chamber $\mathfrak{h}_{\mathbb{R}}^{+}$.
The field $Z$ is gradient, not with respect to the inner product coming from
$\mathfrak{g}$ (the real part of  $\mathcal{H}$), but with respect to the 
Riemannian metric  $m$ on the adjoint orbit $\mathcal{O}\left(
H_{0}\right) $, which does not extend naturally to $\mathfrak{g}$.

The metric $m$ is defined as follows: the tangent space $%
T_{x}\mathcal{O}\left( H_{0}\right) $ is the  image of $\mathrm{ad}\left(
x\right) $, which is the sum of the eigenspaces associated to the  nonzero 
eigenvalues of $x$. This happens because  $\mathrm{ad}\left( x\right) $ 
is conjugate to $\mathrm{ad}\left( H_{0}\right) $ (the formula 
$\mathrm{ad}\left( \phi x\right) =\phi \circ \mathrm{ad}\left( x\right)
\circ \phi ^{-1}$ holds true for any automorphism  $\phi \in \mathrm{Aut}\left( 
\mathfrak{g}\right) $, in particular for $\phi =\mathrm{Ad}\left( g\right) $%
, $g\in G$). Now, $\mathrm{ad}\left( H_{0}\right) $ is diagonalizable
and its image is the sum of the root spaces, which are the eigenspaces 
of the nonzero eigenvalues of  $\mathrm{ad}%
\left( H_{0}\right) $ (since $H_{0}$ is regular). By conjugation
the same is true for $\mathrm{ad}\left( x\right) $, $x\in \mathcal{O}%
\left( H_{0}\right) $.
As a consequence, the restriction of $\mathrm{ad}\left( x\right) $ 
to its image is an invertible linear transformation.

Taking this into account, define 
\[
m_{x}\left( u,v\right) =\left( \mathrm{ad}\left( x\right) ^{-1}u,\mathrm{ad}%
\left( x\right) ^{-1}v\right) , 
\]%
where $\left( \cdot ,\cdot \right) $ is the inner product given by the real 
part of $\mathcal{H}\left( \cdot ,\cdot \right) $, and the inverse of $\mathrm{ad}%
\left( x\right) $ is just the inverse of its restriction to the 
tangent space. The form $m_{x}\left( \cdot ,\cdot \right) $ is a well defined 
Riemannian metric on  $\mathcal{O}\left( H_{0}\right) $.

\begin{remark} The realification $\mathfrak{g}^{\mathbb{R}}$ of $%
	\mathfrak{g}$ is a real semisimple Lie algebra. Its 
	Cartan--Killing form $\langle \cdot ,\cdot \rangle ^{\mathbb{R}}$ is given by $%
	\langle \cdot ,\cdot \rangle ^{\mathbb{R}}=2\Re \langle \cdot ,\cdot
	\rangle $ (see \cite{amalglie}). Consequently, the inner product $\left(
	\cdot ,\cdot \right) $ is given by
	\[
	\left( X,Y\right) =-\frac{1}{2}\langle X,\tau Y\rangle ^{\mathbb{R}} 
	\]%
	where $\tau $ is conjugation with respect to $\mathfrak{u}$,
	which is a linear transformation of  $%
	\mathfrak{g}^{\mathbb{R}}$ (over $\mathbb {R}$).
\end{remark}

Returning to the field $Z\left( x\right) $, define the height function $h_{H}\colon%
\mathcal{O}\left( H_{0}\right) \rightarrow \mathbb{R}$ by
\[
h_{H}\left( x\right) =\left( x,H\right) . 
\]%
Given $A\in \mathfrak{g}$, the  tangent vector $[A,x]$ is given by 
\[
\lbrack A,x]=\frac{d}{dt}_{\left\vert t=0\right. }\mathrm{Ad}\left(
e^{tA}\right) x. 
\]%
Therefore, 
\begin{equation}
\left( dh_{H}\right) _{x}\left( [A,x]\right) =\frac{d}{dt}_{\left\vert
	t=0\right. }\left( \mathrm{Ad}\left( e^{tA}\right) x,H\right) =\left(
[A,x],H\right) .  \label{fordiffcaltur}
\end{equation}%
On one hand, 
\begin{eqnarray*}
	m_{x}\left( [A,x],Z\left( x\right) \right) &=&-m_{x}\left( \mathrm{ad}\left(
	x\right) A,\mathrm{ad}\left( x\right) [\tau x,H]\right) \\
	&=&-\left( A,[\tau x,H]\right) ,
\end{eqnarray*}%
by definition of $m_{x}$. On the other hand, by 
lemma \ref{lemsimetria} below,
$$\left( A,[\tau x,H]\right)
=\left( A,\mathrm{ad}\left( \tau x\right) H\right) =-\left( \mathrm{ad}%
\left( x\right) A,H\right) \text{.}$$ Thus,
$$
m_{x}\left( [A,x],Z\left( x\right) \right) =\left( \mathrm{ad}\left(
x\right) A,H\right) 
=-\left( [A,x],H\right) .
$$
Combining this with (\ref{fordiffcaltur}) we arrive at 
\[
\left( dh_{H}\right) _{x}\left( [A,x]\right) =-m_{x}\left( [A,x],Z\left(
x\right) \right) . 
\]%
In conclusion:

\begin{proposition}
	$Z\left( x\right) =-\grad h_{H}$ with respect to the metric $%
	m_{x}$. \ 
\end{proposition}

\begin{lemma}\label{sym}
	\label{lemsimetria}Consider the inner product $\left( \cdot ,\cdot \right) :=%
	\Re \mathcal{H}\left( \cdot ,\cdot \right) $, then
	
	\begin{itemize}
		\item the conjugation $\tau $  is an isometry for this inner product, and 
		$$\left( \mathrm{ad}\left( X\right)
		Y,Z\right) =-\left( Y,\mathrm{ad}\left( \tau X\right) Z\right) \text{.}$$
		\item  if $\tau X=X$, that is, if $X\in \mathfrak{u}$, then $%
		\mathrm{ad}\left( X\right) $ is antisymmetric for
		$\left( \cdot ,\cdot \right) $, 
		\item  if $\tau Y=-Y$, that is,  $Y\in i\mathfrak{u}$, then $\mathrm{ad}\left( Y\right) $ 
		is symmetric for  $\left( \cdot ,\cdot \right) $.
	\end{itemize}
\end{lemma}

\begin{proof}
	If $X\in \mathfrak{g}$ then $\mathcal{H}\left( \mathrm{ad}\left(
	X\right) Y,Z\right) =-\mathcal{H}\left( Y,\mathrm{ad}\left( \tau X\right)
	Z\right) $, and it follows that the same relation holds true for the inner product  $\left(
	\cdot ,\cdot \right) .$
	In fact, 
	\[
	\mathcal{H}\left( \tau X,Y\right) =-\langle \tau X,\tau Y\rangle =-\langle
	\tau Y,\tau X\rangle =\mathcal{H}\left( \tau Y,X\right) =\overline{\mathcal{H%
		}\left( X,\tau Y\right) }, 
	\]%
	which means that $\left( \tau X,Y\right) =\left( X,\tau Y\right) $.
	For the second item: 
	\begin{eqnarray*}
		\mathcal{H}\left( [X,Y],Z\right) &=&-\langle \lbrack X,Y],\tau Z\rangle
		=\langle Y,[X,\tau Z]\rangle \\
		&=&\langle Y,\tau \lbrack \tau X,Z]\rangle =-\mathcal{H}\left( Y,[\tau
		X,Z]\right) .
	\end{eqnarray*}%
\end{proof}

\begin{remark}({\it $Z$ as a field on $\mathfrak{g}$})
	We show that considered on the entire vector space $\mathfrak{g}$ the vector field  $Z\left( x\right)
	=[x,[\tau x,H]]$ is not gradient with respect to $\left( \cdot
	,\cdot \right) $.
	Take the differential form $\alpha
	_{x}\left( v\right) =\left( v,Z\left( x\right) \right) $. Then $%
	d\alpha \left( v,w\right) =v\alpha \left( w\right) -w\alpha \left( v\right)
	-\alpha \left[ v,w\right] $, where the last term vanishes if $v$ and 
	$w$ are regarded as constant vector fields on $\mathfrak{g}$. The expression
	for  $Z$ then gives
	$
	\left( d\alpha \right) _{x}\left( v,w\right) =$
	$$\left( w,[v,[\tau x,H]]\right)
	+\left( w,[x,[\tau v,H]]\right) -\left( v,[w,[\tau x,H]]\right) -\left(
	v,[x,[\tau w,H]]\right) . 
	$$%
	Evaluating this expression on $x=H_{1}\in \mathfrak{h}$, we obtain
	\begin{eqnarray*}
		\left( d\alpha \right) _{x}\left( v,w\right) &=&\left( w,[H_{1},[\tau
		v,H]]\right) -\left( v,[H_{1},[\tau w,H]]\right) \\
		&=&\left( w,\mathrm{ad}\left( H\right) \mathrm{ad}\left( H_{1}\right) \tau
		v\right) -\left( v,\mathrm{ad}\left( H\right) \mathrm{ad}\left( H_{1}\right)
		\tau w\right) .
	\end{eqnarray*}%
	We have
	$$
	\left( w,\mathrm{ad}\left( H\right) \mathrm{ad}\left( H_{1}\right) \tau
	v\right) =$$
	$$\left( \tau w,\tau \tau ^{-1}\mathrm{ad}\left( H\right) \mathrm{%
		ad}\left( H_{1}\right) \tau v\right) =\left( \tau w,\mathrm{ad}\left( \tau
	H\right) \mathrm{ad}\left( \tau H_{1}\right) v\right) $$
	$$=\left( \mathrm{ad}\left( \tau H_{1}\right) \mathrm{ad}\left( \tau
	H\right) \tau w,v\right) =\left( \mathrm{ad}\left( \tau H\right) \mathrm{ad}%
	\left( \tau H_{1}\right) \tau w,v\right)
	$$
	where the last equality comes from the fact that $\mathrm{ad}\left( H\right) 
	$ commutes with $\mathrm{ad}\left( H_{1}\right) $. Therefore, 
	\[
	\left( d\alpha \right) _{x}\left( v,w\right) =\left( \mathrm{ad}\left( \tau
	H\right) \mathrm{ad}\left( \tau H_{1}\right) \tau w,v\right) -\left( \mathrm{%
		ad}\left( H\right) \mathrm{ad}\left( H_{1}\right) \tau w,v\right) . 
	\]%
	But, setting $\tau H=-H$ and $\tau \left( H_{1}\right) =H_{1}$,  then  the right hand side becomes
	$-2\left( \mathrm{ad}\left( H\right) \mathrm{ad}\left(
	H_{1}\right) \tau w,v\right) $ which  does not vanish identically on $v,w$%
	. Thus, $d\alpha \neq 0$, implying that the vector field is not
	gradient on $\mathfrak g$.
\end{remark}

We return to the study of the singularities of the gradient field $Z$ on the orbit 
$\mathcal{O} \left( H_{0}\right)$.
We have verified that the set of such singularities  is $\mathcal{O}\left( H_{0}\right) \cap \mathfrak{h}
$, which is the orbit of $H_{0}\in \mathfrak{h}$ by the Weyl group.
We now recall the proof that these singularities are nondegenerate. To see this,
let $x=wH_{0}$ be one of the singularities. Then the differential of $Z$ 
at $x$ is given by
\begin{eqnarray*}
	dZ_{x}\left( v\right) &=&[v,[\tau x,H]]+[x,[\tau v,H]]=[x,[\tau v,H]] \\
	&=&-\mathrm{ad}\left( x\right) \mathrm{ad}\left( H\right) \left( \tau
	v\right) .
\end{eqnarray*}%
The tangent space to $\mathcal{O}\left( H_{0}\right) $ at $x$ is
\[
T_{x}\mathcal{O}\left( H_{0}\right) =\sum_{\alpha \in \Pi }\mathfrak{g}%
_{\alpha }=\sum_{\alpha >0}\left( \mathfrak{g}_{\alpha }\oplus \mathfrak{g}%
_{-\alpha }\right) . 
\]%
If $v=\sum_{\alpha \in \Pi }a_{\alpha }X_{\alpha }$ then $\tau
v=-\sum_{\alpha \in \Pi }\overline{a_{\alpha }}X_{-\alpha }$. Consequently, 
\begin{eqnarray*}
	dZ_{x}\left( v\right) &=&\mathrm{ad}\left( x\right) \mathrm{ad}\left(
	H\right) \left( \sum_{\alpha \in \Pi }\overline{a_{\alpha }}X_{-\alpha
	}\right) \\
	&=&\sum_{\alpha \in \Pi }\overline{a_{\alpha }}\alpha \left( x\right) \alpha
	\left( H\right) X_{-\alpha }.
\end{eqnarray*}%
In particular, let $\alpha $ be a \textit{positive} root. Then, $%
\mathfrak{g}_{\alpha }+\mathfrak{g}_{-\alpha }$ (regarded as a 
\textit{real} vector space) is invariant by $dZ_{x}$.
Furthermore, with respect to the basis $%
\{X_{\alpha },X_{-\alpha },iX_{\alpha },iX_{-\alpha }\}$, the restriction of
$dZ_{x}$ to this  subspace is given by the matrix
\[
\alpha \left( x\right) \alpha \left( H\right) \left( 
\begin{array}{cccc}
0 & 1 &  &  \\ 
1 & 0 &  &  \\ 
&  & 0 & -1 \\ 
&  & -1 & 0%
\end{array}%
\right) , 
\]%
which has eigenvalues $\pm \alpha \left( x\right) \alpha \left(
H\right) $ with associated eigenspaces 
\begin{eqnarray*}
	V_{-\alpha \left( x\right) \alpha \left( H\right) } &=&\mathrm{span}_{%
		\mathbb{R}}\{X_{\alpha }-X_{-\alpha },i\left( X_{\alpha }+X_{-\alpha
	}\right) \}=\left( \mathfrak{g}_{\alpha }+\mathfrak{g}_{-\alpha }\right)
	\cap \mathfrak{u}, \\
	V_{\alpha \left( x\right) \alpha \left( H\right) } &=&\mathrm{span}_{\mathbb{%
			R}}\{X_{\alpha }+X_{-\alpha },i\left( X_{\alpha }-X_{-\alpha }\right)
	\}=\left( \mathfrak{g}_{\alpha }+\mathfrak{g}_{-\alpha }\right) \cap i%
	\mathfrak{u}.
\end{eqnarray*}

Therefore, $T_{x}\mathcal{O}\left( H_{0}\right) =\sum_{\alpha \in \Pi }%
\mathfrak{g}_{\alpha }$  decomposes into $T_{x}\mathcal{O}\left(
H_{0}\right) =V_{x}^{+}\oplus V_{x}^{-}$, where $V_{x}^{+}$ (unstable space) 
is the sum of eigenspaces with positive eigenvalues and $%
V_{x}^{-}$ (stable space) is where  $dZ_{x}$ has negative eigenvalues.
The dimension of  $T_{x}\mathcal{O}\left(
H_{0}\right) $ over $\mathbb{R}$ is $2|\Pi |$, whereas $\dim _{\mathbb{R}}V^{\pm
}=|\Pi |$.

\begin{proposition}
	The subspaces $V_{x}^{+}$ and $V_{x}^{-}$ are Lagrangian  with respect to the 
	symplectic form $\Omega =\Im \mathcal{H}$.
\end{proposition}

\begin{proof}
	$V_{\alpha \left( x\right) \alpha \left( H\right) }$ and $%
	V_{-\alpha \left( x\right) \alpha \left( H\right) }$ are isotropic subspaces, since 
	they are contained in either $\mathfrak{u}$ or $i\mathfrak{u}$
	and both are subspaces where the  Hermitian form $\mathcal{H}$ takes 
	real values. On the other hand,  if $\alpha \neq \beta $ are positive roots,
	then $\mathfrak{g}_{\alpha }+\mathfrak{g}_{-\alpha }$ is orthogonal to
	$\mathfrak{g}_{\beta }+\mathfrak{g}_{-\beta }$ with respect to the 
	Cartan--Killing form, and since these subspaces are $\tau $%
	-invariant they are also orthogonal with respect to  $%
	\mathcal{H}$. Therefore, $\mathcal{H}$ assumes real values on $%
	V_{x}^{+}$ and on $V_{x}^{-}$ as well, hence these subspaces are isotropic, and
	by dimension count  they are Lagrangian.
\end{proof}

The subspaces $V_{x}^{+}$ and $V_{x}^{-}$ are the tangent subspaces
to the unstable and stable submanifolds of $Z$ with respect to the 
fixed point  $x$. These submanifolds are denoted by $\mathcal{V}%
_{x}^{+}$ and $\mathcal{V}_{x}^{-}$, respectively.

We will investigate the stable manifold of $Z$  for the  case $x=H_{0}$.
If $\alpha >0$ then $\alpha \left( H_{0}\right) $, $\alpha \left(
H\right) $ and $\alpha \left( H_{0}\right) \alpha \left( H\right) $ are all 
positive. It follows that 
\[
V_{H_{0}}^{+}=i\mathfrak{u}\cap \sum_{\alpha >0}\left( \mathfrak{g}_{\alpha
}+\mathfrak{g}_{-\alpha }\right), \qquad V_{H_{0}}^{-}=\mathfrak{u}\cap
\sum_{\alpha >0}\left( \mathfrak{g}_{\alpha }+\mathfrak{g}_{-\alpha }\right)
. 
\]

%TCIMACRO{\TeXButton{vspace12pt}{\vspace{12pt}}}%
%BeginExpansion
\vspace{12pt}%
%EndExpansion

%TCIMACRO{\TeXButton{noindent}{\noindent}}%
%BeginExpansion
\noindent%
%EndExpansion
%*****\textbf{Conjecture:} $\mathcal{V}_{H_{0}}^{-}=\bigcup_{z\in \mathfrak{%
%		h}_{\mathbb{R}}^{+}}\left( z+V_{H_{0}}^{-}\right) \cap \mathcal{O}\left(
%H_{0}\right) $.
%
%This conjecture is corroborated by the lemma below.  

Observe that $\mathfrak{u}\cap \sum_{\alpha
	>0}\left( \mathfrak{g}_{\alpha }+\mathfrak{g}_{-\alpha }\right) $ is a
\textit{real} vector space with basis
\[
\{A_{\alpha }=X_{\alpha }-X_{-\alpha },iS_{\alpha }=i\left( X_{\alpha
}+X_{-\alpha }\right) :\alpha >0\}. 
\]%
Moreover, for $H\in \mathfrak{h}$ and $\alpha >0$ the following relations hold:

\begin{itemize}
	\item $[H,A_{\alpha }]=\alpha \left( H\right) \left( X_{\alpha }+X_{-\alpha
	}\right) $.
	
	\item $[H,S_{\alpha }]=\alpha \left( H\right) i\left( X_{\alpha }-X_{-\alpha
	}\right) $.
	
	\item $\langle A_{\alpha },S_{\alpha }\rangle =0$, and $\langle A_{\alpha
	},A_{\alpha }\rangle =\langle A_{\alpha },A_{\alpha }\rangle =2$ since,  $\langle X_{\alpha },X_{-\alpha }\rangle =1$.
	
	\item If $\beta \neq \alpha $ then $\langle A_{\alpha },A_{\beta
	}\rangle =\langle S_{\alpha },S_{\beta }\rangle =\langle A_{\alpha
	},S_{\beta }\rangle =\langle S_{\alpha },A_{\beta }\rangle =0$.
\end{itemize}

\begin{lemma}
	For $x=z+y$ with $y\in \mathfrak{u}$ and $z\in \mathfrak{h}_{\mathbb{R}}^{+}$,
	we have that $\mathcal{H}_{\tau }\left( Z\left( x\right) ,y\right) $ is
	real and negative.
\end{lemma}

\begin{proof}
	We have $\tau x=y-z$. Thus, $Z\left( x\right)
	=[y+z,[y-z,H]]=[y+z,[y,H]]$. Consequently, 
	\begin{eqnarray*}
		\left( Z\left( x\right) ,y\right) &=&\left( [y+z,[y,H]],y\right) =\left(
		[y,H],[z,y]\right) \\
		&=&-\left( [H,y],[z,y]\right)\text{.}
	\end{eqnarray*}
	Set $y=\sum_{\alpha >0}\left( a_{\alpha }S_{\alpha }+b_{\alpha
	}A_{\alpha }\right) $ with $a_{\alpha },b_{\alpha }\in \mathbb{R}$. Then,
	by the above relations 
	\[
	\lbrack H,y]=\sum_{\alpha >0}\alpha \left( H\right) \left( a_{\alpha }\left(
	X_{\alpha }-X_{-\alpha }\right) +b_{\alpha }i\left( X_{\alpha }+X_{-\alpha
	}\right) \right) , 
	\]%
	and similarly with $z$ in place of $H$. Still using the above relations, we obtain
	\[
	\langle \lbrack H,y]],[z,y]\rangle =2\sum_{\alpha >0}\alpha \left( H\right)
	\alpha \left( H_{0}\right) \left( a_{\alpha }^{2}+b_{\alpha }^{2}\right) 
	\]%
	which is $>0$ because $\alpha \left( H\right) ,\alpha \left( H_{0}\right) >0$
	and $a_{\alpha },b_{\alpha }\in \mathbb{R}$. This finishes the proof, since $\left( Z\left( x\right) ,y\right) =-\left( [H,y],[z,y]\right) $.
\end{proof}

\section{Lagrangian vanishing cycles}\label{LVC}

We construct Lagrangian spheres inside regular fibres, which  are our candidates for 
vanishing cycles. The correct dimension of the desired spheres is $n-1$ real, 
that is, half  of the dimension of the regular fibre. Here $n$ is the complex dimension of the adjoint orbit, and 
the real dimension of the flag $\mathbb{F}_{\Theta }$  where $\Theta =\Theta
\left( H_{0}\right) =\{\alpha \in \Sigma :\alpha \left( H_{0}\right) =0\}$.
The number of Lagrangian spheres to be found equals  $|\mathcal{W}|$, that is, the number of singularities. 

Here we assume that $H_{0}\in \mathrm{cl}\mathfrak{a}^{+}$  and that  $H\in \mathfrak{a}%
^{+} $, hence $H$ is regular.
Recall that the symplectic form $\Omega $ on the orbit
$\mathcal{O}\left( H_{0}\right) $  is the restriction of the imaginary part of the 
Hermitian form of $\mathfrak{g}$
\begin{equation*}
\mathcal{H}_{\tau }\left( X,Y\right) =-\langle X,\tau Y\rangle .
\end{equation*}%
On the other hand, the real part is the inner product defined by
\begin{equation*}
B_{\tau }\left( X,Y\right) =-\text{Re}\langle X,\tau Y\rangle =-\frac{1}{2}%
\langle X,\tau Y\rangle ^{R},
\end{equation*}%
where   $\langle \cdot ,\cdot \rangle ^{R}$ is the  Cartan--Killing
form of the realification of  $\mathfrak{g}$. Thus, 
\begin{equation*}
\mathcal{H}_{\tau }\left( X,Y\right) =B_{\tau }\left( X,Y\right) +i\Omega
\left( X,Y\right) 
\end{equation*}%
and the equality $\Omega \left( X,Y\right) =B_{\tau }\left( X,iY\right) $ holds
since $\mathcal{H}_{\tau }\left( X,Y\right) $ is Hermitian.

We can  search for an isotropic submanifold by taking  a subspace
$V\subset \mathfrak{g}$ where  $\mathcal{H}_{\tau }$ takes real values, 
and then check whether the intersection  $V\cap \mathfrak{g}$ is indeed a submanifold.

Two examples of subspaces where $\mathcal{H}_{\tau }$ takes real values are:
i) the  compact real form $\mathfrak{u}$, where $\mathcal{H}%
_{\tau }$ is negative definite and ii) the symmetric part  $i\mathfrak{u}$%
, where $\mathcal{H}_{\tau }$ is positive definite.

The intersection  $%
\mathfrak{u}\cap \mathcal{O}\left( H_{0}\right) $ is empty because the 
eigenvalues of  $\mathrm{ad}\left( X\right) $, for $X\in \mathcal{O}\left(
H_{0}\right) $ are real  whereas those of $\mathrm{ad}\left( Y\right) $%
, for  $Y\in \mathfrak{u}$ are imaginary. The latter happens because $\mathrm{ad}\left( Y\right) $
is anti-symmetric with respect to the Cartan--Killing form of $\mathfrak u$, see lemma \ref{sym}.
On the other hand, the intersection $i\mathfrak{u}\cap \mathcal{O}\left(
H_{0}\right) $ is the flag  $\mathbb{F}_{\Theta }$ itself, since it is the  
orbit of the compact group $U=\exp \mathfrak{u}$.

Therefore, $\mathbb{F}%
_{\Theta }$ is an isotropic submanifold, in fact  Lagrangian,  and
any submanifold of  $\mathbb{F}_{\Theta }$ is isotropic as well.
Moreover, the function $f_{H}\left( x\right) =\langle
H,x\rangle $ takes real values on  $\mathbb{F}_{\Theta }=i\mathfrak{u}\cap 
\mathcal{O}\left( H_{0}\right) $. 
Since by hypothesis $H$ is regular, it follows that the restriction 
$f_{H}^{\Theta }$ of $%
f_{H}$ to $\mathbb{F}_{\Theta }$ is a  Morse function. 
The origin  $H_{0}$ is a  singularity and the hypothesis that 
$H_{0}\in \mathrm{cl}\mathfrak{a}^{+}$ implies that  $H_{0}$ is an 
attractor (with negative definite Hessian). Therefore, the levels  $%
\left( f_{H}^{\Theta }\right) ^{-1}\left( f_{H}^{\Theta }\left( x\right)
\right) $ of $f_{H}^{\Theta }$ around $H_{0}$ are codimension 1 
spheres in  $\mathbb{F}_{\Theta }$. These levels are isotropic 
submanifolds. Clearly  $\left( f_{H}^{\Theta
}\right) ^{-1}\left( f_{H}^{\Theta }\left( x\right) \right) \subset \left(
f_{H}\right) ^{-1}\left( f_{H}\left( x\right) \right) $ and since  $\dim \left(
f_{H}\right) ^{-1}\left( f_{H}\left( x\right) \right) =\dim \mathbb{F}%
_{\Theta }-2$, it follows that for $x$ around $%
H_{0}$ the spheres $\left( f_{H}^{\Theta }\right)
^{-1}\left( f_{H}^{\Theta }\left( x\right) \right) $ are  Lagrangian cycles  at the levels  $\left( f_{H}\right)
^{-1}\left( f_{H}\left( x\right) \right) $ of the  Lefschetz fibrations.

We can now carry out the analogous construction around  other 
critical points   $wH_{0}$, $w\in \mathcal{W}$.
We  use the following notation

\begin{enumerate}
	\item[($\iota$)] Given a root  $\alpha >0$, let 
	\begin{equation*}
	\mathfrak{u}_{\alpha }=\left( \mathfrak{g}_{\alpha }\oplus \mathfrak{g}%
	_{-\alpha }\right) \cap \mathfrak{u} \qquad \mathrm{and}\qquad i%
	\mathfrak{u}_{\alpha }=\left( \mathfrak{g}_{\alpha }\oplus \mathfrak{g}%
	_{-\alpha }\right) \cap i\mathfrak{u}.
	\end{equation*}
	Taking a  Weyl basis $X_{\beta }\in \mathfrak{g}_{\beta }$, with $\beta $
	a root, these subspaces are generated by:
	
	\begin{itemize}
		
		\item $\mathfrak{u}_{\alpha }=\mathrm{span}_{\mathbb{R}}\{A_{\alpha
		}=X_{\alpha }-X_{-\alpha },iS_{\alpha }=i\left( X_{\alpha }+X_{-\alpha
		}\right) \}$ and
		
		\item $i\mathfrak{u}_{\alpha }=\mathrm{span}_{\mathbb{R}}\{iA_{\alpha
		}=i\left( X_{\alpha }-X_{-\alpha }\right) ,S_{\alpha }=X_{\alpha
		}+X_{-\alpha}\}$.
	\end{itemize}

	\item[($\iota\iota$)] For $w\in \mathcal{W}$, let  $\Pi _{w}=\Pi ^{+}\cap w^{-1}\Pi ^{-}$  be the 
	set of positive roots that are taken to negative roots by $w$.
	
	\item[($\iota\iota\iota$)] For $w\in \mathcal{W}$ define the real vector subspace 
	\begin{equation*}
	V_{w}=\mathfrak{h}_{\mathbb{R}}\oplus \sum_{\alpha \in \Pi _{w}}\mathfrak{u}%
	_{\alpha }\oplus \sum_{\alpha \in \Pi ^{+}\setminus \Pi _{w}}i\mathfrak{u}%
	_{\alpha }.
	\end{equation*}
\end{enumerate}

When $w=1$ the subspace  $V_{1}=i\mathfrak{u}$, since $\Pi _{1}=\emptyset 
$. The subspaces $V_{w}$, $1\neq w\in \mathcal{W}$, will replace $i%
\mathfrak{u}$ in the constructions of spheres  around the critical points  $%
wH_{0}$.

\begin{lemma}
	\label{lemhrealemvw}   $\mathcal{H}_{\tau }$ and the 
	Cartan--Killing  form $\langle \cdot ,\cdot \rangle $ take real values in  $V_{w}$.
\end{lemma}

\begin{proof} Both
	$\mathcal{H}_{\tau }$ and  $\langle \cdot ,\cdot \rangle $ are
	real in each of the components of  $V_{w}$ (positive definite in  $%
	\mathfrak{h}_{\mathbb{R}}$ and  $\sum_{\alpha \in \Pi ^{+}\setminus \Pi _{w}}i%proposition
	\mathfrak{u}_{\alpha }$ and negative definite in $\sum_{\alpha \in \Pi _{w}}%
	\mathfrak{u}_{\alpha }$). Moreover $\mathfrak{h}_{\mathbb{R}}$, $%
	\mathfrak{u}_{\alpha }$ and  $\mathfrak{u}_{\beta }$ are orthogonal 
	with respect to   $\mathcal{H}_{\tau }$ and to  $\langle \cdot ,\cdot \rangle $ if $%
	\alpha \neq \beta $.\end{proof}

Consequently, the restriction of the imaginary part  $\Omega $
of $H_{\tau }$  to $V_{w}$ vanishes identically.
On the orbit  $\mathcal{O}\left( H_{0}\right) $  we define a distribution
$\Delta _{w}\left( x\right) \subset T_{x}\mathcal{O}\left(
H_{0}\right) $ by
\begin{equation*}
\Delta _{w}\left( x\right) =V_{w}\cap T_{x}\mathcal{O}\left( H_{0}\right) .
\end{equation*}%
By lemma \ref{lemhrealemvw}, the subspaces $\Delta _{w}\left( x\right) $
are isotropic with respect to the symplectic form $%
\Omega $ (restricted to the orbit).
The goal is to prove that this distribution is integrable
(at least around the singularity $wH_{0}$). Once this is accomplished, 
the integral submanifold passing through  $wH_{0}$ will be a 
Lagrangian submanifold (for  $\Omega $). Consequently, a ball around the 
singularity $wH_{0}$, inside the integral submanifold will be 
our candidate to a  Lagrangian thimble.

\begin{remark} A priori a   distribution obtained by intersecting 
	a fixed subspace with the tangent spaces of an embedded submanifold
	(such as  our distribution  $\Delta _{w}$)  might not even 
	be continuous (that is, admit local parametrizations 
	by continuous fields). As an example, consider the case of the circle  $%
	S^{1}=\{x\in \mathbb{R}^{2}:|x|=1\}$.  The horizontal line  $\{\left(
	t,0\right) :t\in \mathbb{R\}}$ contains the tangent space at  $\left(
	0,1\right) $ however, it intersects in dimension zero the tangent spaces  
	of points near  $\left( 0,1\right) $. For a continuous distribution
	the dimension does not decrease around a point. 
\end{remark}

At the singularity $wH_{0}$ (or any other singularity) the 
distribution is 
\begin{equation*}
\Delta _{w}\left( wH_{0}\right) =\sum_{\alpha \in \Pi _{w}}\mathfrak{u}%
_{\alpha }\oplus \sum_{\alpha \in \Pi ^{+}\setminus \Pi _{w}}i\mathfrak{u}%
_{\alpha }.
\end{equation*}%
This is due to the fact that the tangent space at  $wH_{0}$ is given by
\begin{equation*}
T_{wH_{0}}\mathcal{O}\left( H_{0}\right) =\sum_{\alpha \in \Pi }\mathfrak{g}%
_{\alpha }
\end{equation*}%
which intersects  $V_{w}$ at $%
\mathfrak{u}_{\alpha }$ and $i\mathfrak{u}_{\alpha }$,  showing that 
\begin{equation*}
\dim _{\mathbb{R}}\Delta _{w}\left( wH_{0}\right) =\frac{1}{2}\dim _{\mathbb{%
		R}}\mathcal{O}\left( H_{0}\right) .
\end{equation*}%
Hence, $\Delta _{w}\left( wH_{0}\right) $ is a Lagrangian
subspace. It follows that $\dim _{\mathbb{R}}\Delta _{w}\left(
wH_{0}\right) \geq \dim \Delta _{w}\left( x\right) $, $x\in \mathcal{O}%
\left( H_{0}\right) $, since the subspaces  $\Delta _{w}\left(
x\right) $ are isotropic for  $\Omega $. 

We will parametrize   $\Delta _{w}$
around $wH_{0}$ by Hamiltonian fields. 

Given $X\in \mathfrak{g}$ define the real height function   (with respect to the 
inner product $B_{\tau }$) $f_{X}:\mathcal{O}\left( H_{0}\right)
\rightarrow \mathbb{R}$ by
\begin{equation*}
f_{X}\left( x\right) =B_{\tau }\left( X,x\right) .
\end{equation*}%
Denote by $\ham f_{X}$ the Hamiltonian field of  $f_{X}$ with respect to 
$\Omega $ and by  $\grad f_{X}$ its gradient with respect to 
$B_{\tau }$ (both $\Omega $ and $B_{\tau }$ are restricted to  $%
\mathcal{O}\left( H_{0}\right) $). By definition, if  $v\in T_{x}%
\mathcal{O}\left( H_{0}\right) $ then
\begin{equation*}
\left( df_{X}\right) _{x}\left( v\right) =\Omega \left( v,\ham %
f_{X}\left( x\right) \right) =B_{\tau }\left( v,\grad f_{X}\left(
x\right) \right) .
\end{equation*}%
The formula $\Omega \left( X,Y\right) =B_{\tau }\left( X,iY\right) $,
guaranties that 
\begin{equation*}
\Omega \left( v,\ham f_{X}\left( x\right) \right) =B_{\tau }\left( v,i%
\ham f_{X}\left( x\right) \right) =B_{\tau }\left( v,\grad
f_{X}\left( x\right) \right) .
\end{equation*}%
Since this equality holds for all  $v\in T_{x}\mathcal{O}\left( H_{0}\right) 
$ it follows that 
$$\ham f_{X}\left( x\right) =-i\grad f_{X}\left(
x\right) \text{,}$$
for all $x\in \mathcal{O}\left( H_{0}\right) $.

	A basis for  $\Delta _{w}\left( wH_{0}\right) $ is given by   $\ham %
	f_{X}\left( iH_{0}\right) $ with  $X$ belonging to 
	\begin{equation*}
	\{iA_{\alpha },S_{\alpha }:\alpha \in \Pi _{w}\}\cup \{A_{\alpha
	},iS_{\alpha }:\alpha \in \Pi ^{+}\setminus \Pi _{w}\}
	\end{equation*}%
	where  $A_{\alpha }=X_{\alpha }-X_{-\alpha }$ and $S_{\alpha }=X_{\alpha
	}+X_{-\alpha }$. 
	
	Moreover, these Hamiltonian  fields  are tangent to the distribution
	$\Delta _{w}$. 
%
%$V_{w}=\mathfrak{h}_{\mathbb{R}}\oplus \sum_{\alpha \in \Pi _{w}}\mathfrak{u}%
%_{\alpha }\oplus \sum_{\alpha \in \Pi ^{+}\setminus \Pi _{w}}i\mathfrak{u}%
%_{\alpha }.$
%

%..........................................................................................................................

\section{Real Lagrangian thimbles}\label{RLT}

 Let $\mathfrak{u}$ be the 
	Lie algebra of $K$.
We first we recall the following result. 
\begin{proposition}
	\label{propesptangraph}  \cite[Prop 6.2]{GGSM2} The tangent space to $\Gamma%
	\left( k\circ R_{w_{0}}\right) $ at $\left( x,y\right) =\left( x,k\circ
	R_{w_{0}}\left( x\right) \right) $ is given by 
	\begin{equation*}
	\{\left( A,\mathrm{Ad}\left( k\right) A\right) ^{\sim }\left( x,k\circ
	R_{w_{0}}\left( x\right) \right) :A\in \mathfrak{u}\}
	\end{equation*}%
	where $\left( A,\mathrm{Ad}\left( k\right) A\right) ^{\sim }$ is the vector
	field on $\mathbb{F}_{H_{0}}\times \mathbb{F}_{H_{0}^{\ast }}=\mathbb{F}%
	_{\left( H_{0},H_{0}^{\ast }\right) }$ induced by $\left( A,\mathrm{Ad}%
	\left( k\right) A\right) \in \mathfrak{u}\times \mathfrak{u}$.
	\end{proposition}
	
For Lefschetz fibrations on an adjoint orbit  $\mathcal{O}\left( H_{0}\right) $ we can obtain Lagrangian submanifolds as graphs of symplectic maps on the corresponding flag. The idea of our construction is based on the following generalities. Let $f:N\to \mathbb{C}$ be a Lefschetz fibration where the total space $N$ is a Hermitian manifold with Hermitian metric $h$, complex structure $J$, and K\"{a}hler form $\Omega$. Set $%
f=f_{1}+if_{2}$ and let $V$ be a Lagrangian submanifold which 
contains a critical point $x$ of $f$. Let $g=g_{1}+ig_{2}$ be the restriction 
of $f$ to $V$. We define following gradient vector fields
$$ F_{1}=\grad f_{1},\quad F_{2}=\grad f_{2},\quad G_{1}=\grad
	g_{1},\quad \textnormal{and}\quad G_{2}=\grad g_{2}.$$
Since $f$ is a holomorphic function, $df\left( Jv\right) =idf\left(
v\right) $ for all $v\in TN$. This means 
\begin{equation*}
df_{1}\left( Jv\right) +idf_{2}\left( Jv\right) =idf_{1}\left( v\right)
-df_{2}\left( v\right),
\end{equation*}%
hence $df_{2}\left( v\right) =-df_{1}\left( Jv\right) $. That is,
$h\left(F_{2},v\right) =-h\left( F_{1},Jv\right) =h\left(
JF_{1},v\right) $  which shows that
\begin{equation*}
F_{2}=JF_{1}\qquad F_{1}=JF_{2}.
\end{equation*}%
From these equalities it follows that  $x$ is a critical point of $f$ if and only if
$x$ is a critical point of both $f_{1}$ and $f_{2}$. The Hessians of $f_{1}$ and $%
f_{2}$ at the critical point $x$ are related as follows. If $A$ and 
$B$ are vector fields, then
\begin{eqnarray*}
	\mathrm{Hess}f_{1}\left( A,B\right) &=&BAf_{1}=Bh\left( F_{1},A\right) \\
	&=&Bh\left( F_{2},JA\right) =B\left( JA\right) f_{2} \\
	&=&\mathrm{Hess}f_{2}\left( JA,B\right) .
\end{eqnarray*}
If  $f$ has isolated critical points, then both $f_{1}$ and $f_{2}$ are
Morse functions. The relation between the Hessians shows
that at every critical point of $f$ the
number of  positive eigenvalues of the Hessian equals the number of  negative eigenvalues. 
In fact, if $\mathrm{Hess}f_{1}\left( A,A\right) >0$ then
\begin{equation*}
\mathrm{Hess}f_{1}\left( JA,JA\right) =\mathrm{Hess}f_{2}\left( -A,JA\right)
=-\mathrm{Hess}f_{2}\left( JA,A\right) =-\mathrm{Hess}f_{1}\left( A,A\right)
\end{equation*}%
hence the positive definite and negative definite parts have the same dimension.\\

To obtain the relation between  $F_{i}$ and $G_{i}$ we observe that, since
$V$ is a  Lagrangian submanifold, the tangent space to  $N$ at a 
point $y\in V$ decomposes into 
\begin{equation*}
T_{y}N=T_{y}V\oplus JT_{y}V\qquad y\in V,
\end{equation*}%
as $JT_{y}V$ is the orthogonal complement  with respect to the Hermitian
metric $M\left( \cdot ,\cdot \right) $, of $T_{y}V$. Indeed, if $u,v\in
T_{y}V$ then
\begin{equation*}
M\left( u,Jv\right) =-\Omega \left( u,v\right) =0
\end{equation*}%
therefore the subspaces $T_{y}V$ and $JT_{y}V$ are orthogonal
and  have the same dimension, thus are complementary.
Consequently, the following relation between $F_{i}$ and $G_{i}$
holds on points of  $V$.

\begin{proposition}\label{fg}
	If $y\in V$ then $F_{1}\left( y\right) =G_{1}\left( y\right)
	-JG_{2}\left( y\right) $ and $F_{2}\left( y\right) =G_{2}\left( y\right)
	-JG_{1}\left( y\right) $.
\end{proposition}

\begin{proof}
	For the case of  $F_{1}$ take the decomposition $F_{1}\left( y\right)
	=u+Jv\in T_{y}V\oplus JT_{y}V$. Since $g_{1}$ is the restriction of $%
	f_{1}$ to $V$, it follows that $\left( df_{1}\right) _{y}\left( w\right) =\left(
	dg_{1}\right) _{y}\left( w\right) $ if $w\in T_{y}V$. Therefore, for $w\in
	T_{y}V$ 
	\begin{eqnarray*}
		\left( dg_{1}\right) _{y}\left( w\right) &=&\left( df_{1}\right) _{y}\left(
		w\right) =M\left( F_{1}\left( y\right) ,w\right) \\
		&=&M\left( u+v,w\right) =M\left( u,w\right)
	\end{eqnarray*}%
	and we see that $u=G_{1}\left( y\right) $. Now take $Jw\in JT_{y}V$.
	So, 
	\begin{eqnarray*}
		\left( df_{1}\right) _{y}\left( Jw\right) &=&-\left( df_{2}\right)
		_{y}\left( w\right) =-\left( dg_{2}\right) _{y}\left( w\right) \\
		&=&-M\left( G_{2}\left( y\right) ,w\right)
	\end{eqnarray*}%
	and $M\left( F_{1}\left( y\right) ,Jw\right) =-M\left(
	G_{2}\left( y\right) ,w\right) $, that is, $$M\left( v,w\right) =M\left(
	Jv,Jw\right) =-M\left( G_{2}\left( y\right) ,w\right) \text{.}$$ Since $w$ is
	arbitrary, it follows that $v=-G_{2}\left( y\right) $.
	Thus, for $y\in V$\\
	$$F_{2}=-JF_{1}=-J\left( G_{1}-JG_{2}\right) =G_{2}-JG_{1}\text{.}$$ 
\end{proof}

The expressions from  proposition \ref{fg}  show that if $G_{2}=0$, then
$F_{1}=G_{1}$ and, consequently $F_{1}$ is tangent to $V$. It follows that

\begin{corollary}
	If the imaginary part is constant on the  Lagrangian subvariety $V$, then
	$\grad f_{1}$ is tangent to $V$.
\end{corollary}

Consequently, we obtain the following method of constructing stable and 
unstable manifolds of  $\grad f_{1}$ at a critical point $x$ (in
the case of Morse functions).

\begin{proposition}
	\label{proplagrandest}Let $V$ be a Lagrangian submanifold that contains 
	a critical point  $x$ of the function $f=f_{1}+if_{2}$ that defines a 
	Lefschetz fibration. Suppose that  $f_{2}$ is constant on  $V$ and that 
	the restriction of the  Hessian $\mathrm{Hess}\left( f\right) \left( x\right) $
	to the tangent subspace $T_{x}V$ is negative  definite (respectively
	positive definite). Then, the stable (respectively unstable)
	manifold of $g_{1}$ in $V$ is an open subset of the stable 
	(respectively unstable) manifold of  $f_{1}$.
\end{proposition}

\begin{proof}
	The Hessian of $g_{1}$ is the restriction to  $T_{x}V$ of the Hessian of
	$f_{1}$. The hypothesis guaranties that  the fixed point $x$ is an attractor
	(respectively repeller) of $G_{1}=\grad g_{1}$. Consequently, in the 
	negative definite case, the stable manifold of $G_{1}$ is and open subset  
	$V$ that contains $x$. In this open set $F_{1}$ coincides with $G_{1}$, since
	by hypothesis $f_{2}$ is constant on $V$, that is, $G_{2}=0$.
	Therefore, the stable manifold of $G_{1}$ is contained in the stable manifold of 
	$F_{1}$. A similar argument handles the positive definite case.
\end{proof}

Since the levels of a Morse function in the neighborhood of an 
attracting or repelling singularity are spheres (by the 
Morse lemma), this proposition has the following consequence.

\begin{corollary}
	In the setup of proposition \ref{proplagrandest} consider a level
	$g_{1}^{-1}\{c\}=f_{1}^{-1}\left( c\right) \cap V$ with $c$ 
	near $g_{1}\left( x\right) =f_{1}\left( x\right) $ such that  $c<g_{1}\left(
	x\right) $ in the negative definite case and $c>g\left( x\right) $ in the
	positive definite case. Then, $g_{1}^{-1}\{c\}$ is a sphere of dimension
	$\dim V-1$.
\end{corollary}

The sphere  $g_{1}^{-1}\{c\}$ in this corollary is a 
Lagrangian  submanifold of the level $f^{-1}\{c\}$ (since in proposition  \ref%
{proplagrandest} we took the hypothesis that  $g_{2}$ is  constant).

The next goal is to construct a Lagrangian thimble having as boundary the 
sphere $g_{1}^{-1}\{c\}$ contained in the Lagrangian submanifold $V$. For this
observe that for any $y\in N$,  the symplectic orthogonal of the fibre 
$\Phi _{y}=f^{-1}\{f\left( y\right) \}$ is generated by $F_{1}=
\grad f_{1}$ and $J\grad f_{1}=-\grad f_{2}=-F_{2}$. In fact, $%
F_{1}\left( y\right) $ is the metric orthogonal of     $T_{y}\Phi _{y}$ since
it is  gradient. However, $\Phi _{y}$ is a complex submanifold, thus
$\Omega \left( F_{1}\left( y\right) ,v\right) =M\left(
F_{1}\left( y\right) ,Jv\right) =0$ if $v\in T_{y}\Phi _{y}$. It follows that
\begin{equation*}
\Omega \left( JF_{1}\left( y\right) ,v\right) =M\left( JF_{1}\left( y\right)
,Jv\right) =M\left( F_{1}\left( y\right) ,v\right) =0
\end{equation*}%
if $v\in T_{y}\Phi _{y}$, which shows that  $F_{1}\left( y\right) $ and $%
JF_{1}\left( y\right) $ generate the  symplectic orthogonal to $\Phi _{y}$.

Consequently we obtain the following Lagrangian thimble for  $f$.

\begin{theorem}\label{stable/unstable}
	In the setup of proposition \ref{proplagrandest} take $c$ 
	near $f_{1}\left( x\right) =g_{1}\left( x\right) $. We have that
	\begin{equation*}
	g_{1}^{-1}\left[ c,g_{1}\left( x\right) \right] =f_{1}^{-1}\left[
	c,f_{1}\left( x\right) \right] \cap V \quad \text{in the negative definite case, or } 
	\end{equation*}%
	  \begin{equation*}g_{1}^{-1}\left[ g_{1}\left( x\right) ,c%
	\right] =f_{1}^{-1}\left[ f_{1}\left( x\right) ,c\right] \cap V \quad \text{in the 
	positive definite case}\phantom{xxx} \end{equation*}
	 is homeomorphic to a closed ball in  $\mathbb{R}^{\dim
		V}.$ This ball is a  Lagrangian thimble.
\end{theorem}

\begin{proof}
	In the negative definite case  $g_{1}^{-1}[c,g_{1}\left( x\right) ]$ is the
	Lagrangian thimble obtained by parallel transport of the 
	Lagrangian sphere $g^{-1}\{c\}$ along the line segment $[c,g\left( x\right)
	]\subset \mathbb{R}$. In fact, if $s\in \lbrack c,g\left( x\right) ]$ and $%
	z\in g^{-1}\{s\}$ then the horizontal lift of the vector $d/dt$ is 
	a multiple of  $F_{1}\left( z\right) $. This happens because the 
	horizontal lift is a vector $W=aF_{1}\left( z\right) +bJF_{1}\left( z\right) $%
	, $a,b\in \mathbb{R}$, which satisfies $df_{z}\left( W\right) =\left(
	df_{1}\right) _{z}\left( W\right) +i\left( df_{2}\right) _{z}\left( W\right)
	=d/dt$, thus, $df_{z}\left( W\right) $ is real and therefore coincides with
	$\left( df_{1}\right) _{z}\left( W\right) $. This implies 
	$\left( df_{2}\right) _{z}\left( W\right) =0$, so, 
	\begin{eqnarray*}
		0 &=&M\left( F_{2},W\right) =-M\left( JF_{1}\left( z\right) ,aF_{1}\left(
		z\right) +bJF_{1}\left( z\right) \right) \\
		&=&-bM\left( JF_{1}\left( z\right) ,JF_{1}\left( z\right) \right)
	\end{eqnarray*}%
	consequently, $b=0$. In the negative definite case we have the coefficient $a>0$, since $%
	f_{1}$ grows in the direction of $F_{1}$.
	
	Therefore, the parallel transport of a point of  $g_{1}^{-1}\{c\}$ along 
	the segment $\left[ c,g_{1}\left( x\right) \right] $ follows the trajectories of 
	$F_{1}$ (reparametrized). Such trajectories converge to  $%
	x $, thus, the union of parallel transports of  $s\in \lbrack
	c,g\left( x\right) ]$ is the ball $g_{1}^{-1}\left[ c,g_{1}\left( x\right) %
	\right] =f_{1}^{-1}\left[ c,f_{1}\left( x\right) \right] \cap V$.
	
	The same argument works in the positive definite case,  with $-F_{1}$ in place of  $%
	F_{1}$.
\end{proof}

\begin{definition}\label{rlt}
	A Lagrangian thimble  inside  a stable or unstable of submanifold constructed 
	as in theorem \ref{stable/unstable} is  called a
	\textbf{real Lagrangian thimble}, since it is obtained by lifting of a real horizontal curve.
\end{definition}

\section{The potential and graphs}\label{PG}

The goal in this section is to analyze the behavior of the potential given by the height function  $%
f_{H}\left( x\right) =\langle x,H\rangle $ on   Lagrangian graphs. The
cases of interest here are the graphs of the composites  $m\circ
R_{w_{0}}$ with  $m$ in the torus $T=\exp \left( i\mathfrak{h}_{\mathbb{R}}\right) $%
.  Such  graphs all pass through the critical points of   $f_{H}$. In fact, 
in the product  $\mathbb{F}_{H_{0}}\times \mathbb{F}_{H_{0}^{\ast }}$ these
critical points are given by  $\left( wH_{0},ww_{0}H_{0}^{\ast }\right)
=\left( wH_{0},-wH_{0}\right) $. Since
\begin{equation*}
m\circ R_{w_{0}}\left( wH_{0}\right) =\mathrm{Ad}\left( m\right) \left(
ww_{0}H_{0}^{\ast }\right) =ww_{0}H_{0}^{\ast }
\end{equation*}%
we see that these pairs belong to  $\Gamma\left( m\circ
R_{w_{0}}\right) $.

The Hessian of $ f_{H} $ at a critical point is calculated
considering everything from the point of view of the adjoint orbit $\mathcal{O%
}\left( H_{0}\right) =\mathrm{Ad}\left( G\right) H_{0}$. In this case the field $%
\widetilde{A}$ induced by $A\in \mathfrak{g}$ is linear $\widetilde{A}=%
\mathrm{ad}\left( A\right) $. Therefore $\widetilde{A}f_{H}\left( x\right)
=\langle \lbrack A,x],H\rangle $ and the second derivative is  $\widetilde{B}%
\widetilde{A}f_{H}\left( x\right) =\langle \lbrack A,[B,x]],H\rangle $.
Thus, if $x=wH_{0}$ is a  critical point, then%
\begin{equation}
\mathrm{Hess}\left( f_{H}\right) \left( \widetilde{A}\left( x\right) ,%
\widetilde{B}\left( x\right) \right) =-\langle \lbrack
B,wH_{0}],[A,H]\rangle =-\langle \lbrack wH_{0},B],[H,A]\rangle .
\label{forhesswhzero}
\end{equation}

The goal now is to find the restriction of this Hessian to the tangent spaces
to the graphs  $\Gamma%
\left( m\circ R_{w_{0}}\right) $, $m\in T$ at the critical points. These tangent spaces 
were described in proposition \ref{propesptangraph} using the realization of the homogenous space as
an orbit inside the product $\mathbb{F%
}_{H_{0}}\times \mathbb{F}_{H_{0}^{\ast }}=\mathbb{F}_{\left(
	H_{0},H_{0}^{\ast }\right) }$. Such description must be translated to the viewpoint where
the homogeneous space is the adjoint orbit 
$\mathcal{O}\left( H_{0}\right) =\mathrm{Ad}\left( G\right) H_{0}$.
% , since $\mathrm{Hess}\left( f_{H}\right) $ is calculated at this adjoint orbit.
This translation will be made in the next  proposition.
First recall that from the point of view of the open orbit 
$G\cdot \left( H_{0},-H_{0}\right) \subset \mathbb{F}%
_{H_{0}}\times \mathbb{F}_{H_{0}^{\ast }}$ the critical points are $%
\left( wH_{0},ww_{0}H_{0}^{\ast }\right) =\left( wH_{0},-wH_{0}\right) $, $%
w\in \mathcal{W}$.

\begin{proposition}
	Let $m\in T=\exp \left( i\mathfrak{h}_{\mathbb{R}}\right) $ and consider $%
	\Gamma\left( m\circ R_{w_{0}}\right) $ as a 
	Lagrangian submanifold of  $\mathcal{O}\left( H_{0}\right) =\mathrm{Ad}\left( G\right)
	\cdot H_{0}$. Then the tangent space to $\Gamma\left( m\circ
	R_{w_{0}}\right) $ at the critical point $wH_{0}$, $w\in \mathcal{W}$, is
	generated by the vectors
	
	\begin{enumerate}
		\item $\widetilde{X}_{\alpha }\left( wH_{0}\right) -\mathrm{Ad}
			\left( m\right) \widetilde X_{-\alpha }\left( wH_{0}\right) =[X_{\alpha },wH_{0}]-[%
		\mathrm{Ad}\left( m\right) X_{-\alpha },wH_{0}]$ with $\alpha \left(
		wH_{0}\right) <0$ and
		\item $\widetilde{iX}_{\alpha }\left( wH_{0}\right) +{\mathrm{Ad}%
			\left( m\right) \widetilde{iX}_{-\alpha }}\left( wH_{0}\right) =i[X_{\alpha },wH_{0}]+i[%
		\mathrm{Ad}\left( m\right) X_{-\alpha },wH_{0}]$ with $\alpha \left(
		wH_{0}\right) <0$.
	\end{enumerate}
\end{proposition}

\begin{proof}
	By proposition \ref{propesptangraph} the tangent space to $%
	\Gamma\left( m\circ R_{w_{0}}\right) $ at the  critical point $\left(
	wH_{0},-wH_{0}\right) $ (seen as the orbit in the product) is generated by
	\begin{equation*}
	\left( A,\mathrm{Ad}\left( m\right) A\right) ^{\sim }\left(
	wH_{0},-wH_{0}\right) =\left( \widetilde{A}\left( wH_{0}\right) ,{%
		\mathrm{Ad}\left( m\right) \widetilde A}\left( -wH_{0}\right) \right)
	\end{equation*}%
	with $A\in \mathfrak{u}$.
	
	The real compact form  $\mathfrak{u}$ is generated by $i\mathfrak{h}_{%
		\mathbb{R}}$, $A_{\alpha }=X_{\alpha }-X_{-\alpha }$ and $Z_{\alpha }=i\left(
	X_{\alpha }+X_{-\alpha }\right) $ with $\alpha $ running through all roots.
	The field induced by an element of $i\mathfrak{h}_{\mathbb{R}}$
	vanishes at the critical point  $wH_{0}$ hence it suffices to consider the 
	fields induced by $A_{\alpha }$ and $Z_{\alpha }$.
	
	Choose a root  $\alpha $ such that $\alpha \left( wH_{0}\right) <0$. Then,
	in $\mathbb{F}_{H_{0}}$, $\widetilde{A}_{\alpha }\left( wH_{0}\right) =%
	\widetilde{X}_{\alpha }\left( wH_{0}\right) $ and $\widetilde{Z}_{\alpha
	}\left( wH_{0}\right) =\widetilde{iX}_{\alpha }\left( wH_{0}\right) $ since $%
	\widetilde{X}_{-\alpha }\left( wH_{0}\right) =0$.
	
	On the other hand, $\mathrm{Ad}\left( m\right) A_{\alpha }=\mathrm{Ad}\left(
	m\right) X_{\alpha }-\mathrm{Ad}\left( m\right) X_{-\alpha }$ and $\mathrm{Ad}%
	\left( m\right) Z_{\alpha }=\mathrm{Ad}\left( m\right) iX_{\alpha }+\mathrm{%
		Ad}\left( m\right) iX_{-\alpha }$ given that both $\mathrm{Ad}\left(
	m\right) X_{\pm \alpha }$ and $\mathrm{Ad}\left( m\right) iX_{\pm \alpha
	} $ belong to  $\mathfrak{g}_{\pm \alpha }$ since $\mathrm{Ad}\left(
	m\right) \mathfrak{g}_{\pm \alpha }=\mathfrak{g}_{\pm \alpha }$ (because $%
	m\in T$).
	
	Taking now the induced field on $\mathbb{F}_{H_{0}^{\ast }}$ and using 
	the fact that $\alpha \left( wH_{0}\right) <0$  we obtain that $\mathrm{Ad}\left( m\right) \widetilde{X}_{\alpha
	}\left( -wH_{0}\right) =0$ on $\mathbb{F}%
	_{H_{0}^{\ast }}$ (since $\alpha \left( -wH_{0}\right) >0$).
	Therefore $\mathrm{Ad}\left( m\right) \widetilde{A}_{\alpha }\left(
	-wH_{0}\right) =-\mathrm{Ad}\left( m\right) \widetilde{X}_{-\alpha }\left(
	-wH_{0}\right) $ and $\mathrm{Ad}\left( m\right) \widetilde{Z}_{\alpha }\left(
	-wH_{0}\right) =i\mathrm{Ad}\left( m\right) \widetilde{X}_{-\alpha }\left(
	-wH_{0}\right) $.
	
	Now, the isomorphism between $G\cdot \left( H_{0},-H_{0}\right) $ and $\mathcal{O%
	}\left( H_{0}\right) $ takes a field induced by an element of $\mathfrak{u}$ to an 
	induced field. Moreover, the isomorphism associates $\left(
	wH_{0},-wH_{0}\right) \in \mathbb{F}_{H_{0}}\times \mathbb{F}_{H_{0}^{\ast
	}} $ to $wH_{0}\in \mathcal{O}\left( H_{0}\right) $. This way, the
	isomorphism takes $\left( \widetilde{A}\left( wH_{0}\right) ,
		\mathrm{Ad}\left( m\right) \widetilde{A}\left( -wH_{0}\right) \right) $ to $\widetilde{A}%
	\left( wH_{0}\right) +\mathrm{Ad}\left( m\right) \widetilde{A}\left(
	wH_{0}\right) $ (where the former  $\widetilde{\cdot }$ means the field induced on
	$\mathbb{F}_{H_{0}}$ and $\mathbb{F}_{H_{0}^{\ast }}$ whereas the 
	latter the one induced on $\mathcal{O}\left( H_{0}\right) $).
	
	Hence the tangent space at the critical point  $wH_{0}\in \mathcal{O}%
	\left( H_{0}\right) $ is generated by $\widetilde{X}_{\alpha }\left(
	wH_{0}\right) -\mathrm{Ad}\left( m\right) \widetilde{X}_{-\alpha }\left(
	wH_{0}\right) $ and $\widetilde{iX}_{\alpha }\left( wH_{0}\right) +
	i\mathrm{Ad}\left( m\right) \widetilde{X}_{-\alpha }\left( wH_{0}\right) $. \end{proof}

The generators of the tangent space at  $\Gamma\left( m\circ
R_{w_{0}}\right) $ of the previous proposition can also be described
in the following simpler manner. Take $H_{1}\in \mathfrak{h}_{\mathbb{%
		R}}$ such that $m=e^{iH_{1}}$. Then, $\mathrm{Ad}\left( m\right)
X_{\alpha }=e^{i\alpha \left( H_{1}\right) }X_{\alpha }$. This way, the vector fields 
that provide the generators at $wH_{0}$ become

\begin{itemize}
	\item $\widetilde{X}_{\alpha }-\mathrm{Ad}\left( m\right)
		\widetilde{X}_{-\alpha }=\widetilde{X}_{\alpha }-e^{-i\alpha \left( H_{1}\right) }%
	\widetilde{X}_{-\alpha }$ with $\alpha \left( wH_{0}\right) <0$ and
	
	\item $\widetilde{iX}_{\alpha }+\mathrm{Ad}\left( m\right)
		\widetilde{iX}_{-\alpha }=\widetilde{iX}_{\alpha }+e^{-i\alpha \left( H_{1}\right) }%
	\widetilde{iX}_{-\alpha }$ with $\alpha \left( wH_{0}\right) <0$.
\end{itemize}

It is now possible to calculate $\mathrm{Hess}\left( f_{H}\right) $ at  
 $wH_{0}$ using  formula  (\ref{forhesswhzero}). The
elements of  $\mathfrak{g}$ which define the generating fields belong to  $%
\mathfrak{g}_{\alpha }\oplus \mathfrak{g}_{-\alpha }$. Hence the 
Hessian  vanishes at a pair of generators coming from distinct roots,
since, with respect to the  Cartan-Killing form, $\mathfrak{g}_{\pm
	\alpha }$ is orthogonal to  $\mathfrak{g}_{\pm \beta }$ if $\beta \neq \pm
\alpha $. For the fields provided by a root $\alpha $ with $\alpha
\left( wH_{0}\right) <0$, we obtain (in $wH_{0}$):

\begin{itemize}
	\item $\mathrm{Hess}\left( f_{H}\right) \left( \widetilde{X}_{\alpha
	}-e^{-i\alpha \left( H_{1}\right) }\widetilde{X}_{-\alpha },\widetilde{X}%
	_{\alpha }-e^{-i\alpha \left( H_{1}\right) }\widetilde{X}_{-\alpha }\right)
	=$
	$$-\langle \lbrack wH_{0},X_{\alpha }-e^{-i\alpha \left( H_{1}\right)
	}X_{-\alpha }],[H,X_{\alpha }-e^{-i\alpha \left( H_{1}\right) }X_{-\alpha
	}]\rangle \text{.}$$ The second term equals
	\begin{equation*}
	-\langle \alpha \left( wH_{0}\right) X_{\alpha }+e^{-i\alpha \left(
		H_{1}\right) }\alpha \left( wH_{0}\right) X_{-\alpha },\alpha \left(
	H\right) X_{\alpha }+e^{-i\alpha \left( H_{1}\right) }\alpha \left( H\right)
	X_{-\alpha }\rangle .
	\end{equation*}%
Now, $\langle X_{\alpha },X_{\alpha }\rangle =\langle X_{-\alpha
	},X_{-\alpha }\rangle =0$ and $\langle X_{\alpha },X_{-\alpha }\rangle =1$
	(Weyl basis) therefore, the Hessian becomes
	\begin{equation*}
	-2\alpha \left( wH_{0}\right) \alpha \left( H\right) e^{-i\alpha \left(
		H_{1}\right) }.
	\end{equation*}
	
	\item $\mathrm{Hess}\left( f_{H}\right) \left( \widetilde{iX}_{\alpha
	}+e^{-i\alpha \left( H_{1}\right) }\widetilde{iX}_{-\alpha },\widetilde{iX}%
	_{\alpha }+e^{-i\alpha \left( H_{1}\right) }\widetilde{iX}_{-\alpha }\right)
	=$
	$$-\langle \lbrack wH_{0},iX_{\alpha }+e^{-i\alpha \left( H_{1}\right)
	}iX_{-\alpha }],[H,iX_{\alpha }+e^{-i\alpha \left( H_{1}\right) }iX_{-\alpha
	}]\rangle .$$ That is,%
	\begin{equation*}
	\langle \alpha \left( wH_{0}\right) X_{\alpha }-e^{-i\alpha \left(
		H_{1}\right) }\alpha \left( wH_{0}\right) X_{-\alpha },\alpha \left(
	H\right) X_{\alpha }-e^{-i\alpha \left( H_{1}\right) }\alpha \left( H\right)
	X_{-\alpha }\rangle .
	\end{equation*}%
	Thus the  Hessian equals %
	\begin{equation*}
	-2\alpha \left( wH_{0}\right) \alpha \left( H\right) e^{-i\alpha \left(
		H_{1}\right) }.
	\end{equation*}
	
	\item $\mathrm{Hess}\left( f_{H}\right) \left( \widetilde{X}_{\alpha
	}-e^{-i\alpha \left( H_{1}\right) }\widetilde{X}_{-\alpha },\widetilde{iX}%
	_{\alpha }+e^{-i\alpha \left( H_{1}\right) }\widetilde{iX}_{-\alpha }\right)=$
	$$-\langle \lbrack wH_{0},X_{\alpha }-e^{-i\alpha \left( H_{1}\right)
	}X_{-\alpha }],[H,iX_{\alpha }+e^{-i\alpha \left( H_{1}\right) }iX_{-\alpha
	}]\rangle .$$ That is, 
	\begin{equation*}
	-i\langle \alpha \left( wH_{0}\right) X_{\alpha }+e^{-i\alpha \left(
		H_{1}\right) }\alpha \left( wH_{0}\right) X_{-\alpha },\alpha \left(
	H\right) X_{\alpha }-e^{-i\alpha \left( H_{1}\right) }\alpha \left( H\right)
	X_{-\alpha }\rangle =0.
	\end{equation*}
\end{itemize}

Summing up,

\begin{proposition}
	\label{prophess} The Hessian of $f_{H}$ restricted to the tangent space 
	$$%
	T_{wH_{0}}\left( \Gamma\left( e^{iH_{1}}\circ R_{w_{0}}\right)
	\right) $$ is diagonalizable in the basis 
	\begin{equation*}
	\{\left( \widetilde{X}_{\alpha }-e^{-i\alpha \left( H_{1}\right) }\widetilde{%
		X}_{-\alpha }\right) \left( wH_{0}\right) ,\left( \widetilde{iX}_{\alpha
	}+e^{-i\alpha \left( H_{1}\right) }\widetilde{iX}_{-\alpha }\right) \left(
	wH_{0}\right) :\alpha \left( wH_{0}\right) <0\}.
	\end{equation*}%
	The diagonal elements are given by 
	\begin{equation*}
	-2\alpha \left( wH_{0}\right) \alpha \left( H\right) e^{-i\alpha \left(
		H_{1}\right) }.
	\end{equation*}
\end{proposition}

For example, the orbit of the compact group (the zero section in the identification with the 
cotangent bundle) is $\Gamma\left(
R_{w_{0}}\right) $. If $w=1$ and $H_{1}=0$, then  $-2\alpha \left(
wH_{0}\right) \alpha \left( H\right) e^{-i\alpha \left( H_{1}\right)
}=-2\alpha \left( H_{0}\right) \alpha \left( H\right) $ which is $<0$ since $%
\alpha \left( H_{0}\right) <0$ implies that $\alpha <0$ and consequently $\alpha
\left( H\right) <0$. That is, the Hessian is negative definite, which was to be 
expected  given that the critical point $H_{0}$ is a maximum of  $%
\text{Re}f_{H}$ on the zero section.

%\section{Real thimbles}

 We now consider graphs in the compactification
 $\mathbb{F}_{H_{\protect\mu }}\times \mathbb{F}_{H_{\protect\mu }}^{\ast }$
The isomorphism between the open orbit in $\mathbb{F}_{H_{\mu }}\times 
\mathbb{F}_{H_{\mu }^{\ast }}$ (diagonal action) and the orbit $G\cdot
\left( v_{0}\otimes \varepsilon _{0}\right) $ of $v_{0}\otimes \varepsilon
_{0}\in V\otimes V^{\ast }$ (representation of $G$) leads to a convenient
description of the intersection of graphs of anti-holomorphic functions $%
\mathbb{F}_{H_{\mu }}\rightarrow \mathbb{F}_{H_{\mu }^{\ast }}$ with the
open orbit.

We return to the anti-holomorphic functions $m\circ
R_{w_{0}}\colon {\mathbb F}_{H_{\mu }}\rightarrow \mathbb{F}_{H_{\mu }^{\ast }}$
with $m\in T$, the maximal torus. The submanifold determined by $\mathrm{%
	graph}\left( R_{w_{0}}\right) $ in $R_{w_{0}}$ on $\mathbb{F}_{H_{\mu
}}\times \mathbb{F}_{H_{\mu }^{\ast }}$ is the orbit of the compact group $K$
through $\left( v_{0},\varepsilon _{0}\right) $. This orbit stays inside $%
G\cdot \left( v_{0},\varepsilon _{0}\right) $ and is identified with the $K$%
-orbit of $v_{0}\otimes \varepsilon _{0}$ in $V\otimes V^{\ast }$ (by
equivariance). The isomorphism with the adjoint orbit $\mathrm{Ad}\left(
G\right) H_{\mu }$ associates this $K$-orbit inside $V\otimes V^{\ast }$
with the intersection \ $i\mathfrak{u}\cap $\textrm{Ad}$\left( G\right)
H_{\mu }$ (the Hermitian matrices in the case of $\mathfrak{sl}\left( n+1,%
\mathbb{C}\right) $ or else the zero section of $T^{\ast }\mathbb{F}_{H_{\mu
}}$). This set is formed by the elements $v\otimes \varepsilon \in G\cdot
\left( v_{0}\otimes \varepsilon _{0}\right) $ such that $\ker \varepsilon
=v^{\bot }$ (with respect to the $K$-invariant Hermitian form $\left( \cdot
,\cdot \right) ^{\mu }$ ), since $u\in K$ is an isometry of $\left( \cdot
,\cdot \right) ^{\mu }$ and $\ker \varepsilon _{0}=v_{0}^{\bot }$. The
converse is true as well: if $v\otimes \varepsilon \in G\cdot \left(
v_{0}\otimes \varepsilon _{0}\right) $ and $\ker \varepsilon =v^{\bot }$
then $v\otimes \varepsilon \in \Gamma\left( R_{w_{0}}\right) $. In
fact, if $\ker \varepsilon =v^{\bot }$ and $X\in \mathfrak{u}$ then $\rho
_{\mu }\left( X\right) $ is anti-Hermitian, thus $\left( \rho _{\mu }\left(
X\right) v,v\right) ^{\mu }$ is purely imaginary and since $\ker \varepsilon
=v^{\bot }$, then $\varepsilon \left( \rho _{\mu }\left( X\right) v\right) $
is purely imaginary as well. Therefore, $\langle M\left( v\otimes
\varepsilon \right) ,X\rangle =\varepsilon \left( \rho _{\mu }\left(
X\right) v\right) $ is imaginary for arbitrary $X\in \mathfrak{u}$, which
implies that $M\left( v\otimes \varepsilon \right) \in i\mathfrak{u}$.

Summing up, we obtain the following description of $\Gamma\left(
R_{w_{0}}\right) $ regarded as a subset of $G\cdot \left( v_{0}\otimes
\varepsilon _{0}\right) $. Consider $\Phi ^{-1}\left( \Gamma\left(
R_{w_{0}}\right) \right) \subset G\cdot \left( v_{0}\otimes \varepsilon
_{0}\right) $, which, abusing notation, we also denote by $\Gamma%
\left( R_{w_{0}}\right) $. We have:

\begin{proposition}
	$\Gamma\left( R_{w_{0}}\right) =\{v\otimes \varepsilon \in G\cdot
	\left( v_{0}\otimes \varepsilon _{0}\right) :\ker \varepsilon =v^{\bot }\}.$ 
	%\end{equation*}
\end{proposition}

Consider now the graph of $m\circ R_{w_{0}}\colon \mathbb{F}_{H_{\mu }}\rightarrow 
\mathbb{F}_{H_{\mu }^{\ast }}$ with $m\in T$. In general $\Gamma%
\left( m\circ R_{w_{0}}\right) \subset \mathbb{F}_{H_{\mu }}\times \mathbb{F}%
_{H_{\mu }^{\ast }}$ is not contained in the open orbit and, consequently,
intercepts this orbit in a noncompact subset. In either case, take the subgroup 
\begin{equation*}
U^{m}=\{\left( u,mum^{-1}\right) \in U\times U:u\in U\}.
\end{equation*}%
The graph $\Gamma\left( m\circ R_{w_{0}}\right) $ is the orbit of $K^{m} 
$ through $\left( v_{0},\varepsilon _{0}\right) $. This happens because, if $%
x=u\cdot v_{0}\in \mathbb{F}_{H_{\mu }}$ then $R_{w_{0}}\left( x\right)
=u\cdot \varepsilon _{0}$ therefore 
\begin{equation*}
\left( x,m\circ R_{w_{0}}\left( x\right) \right) =\left( x,m\cdot
u\varepsilon _{0}\right) .
\end{equation*}%
This means that $\Gamma\left( m\circ R_{w_{0}}\right) $ is formed by
elements of the form $\left( x,my\right) $ with $\left( x,y\right) \in 
\Gamma\left( R_{w_{0}}\right) $, that is, 
\begin{equation*}
\Gamma\left( m\circ R_{w_{0}}\right) =m_{2}\left( \Gamma%
\left( R_{w_{0}}\right) \right)
\end{equation*}%
where $m_{2}\left( x,y\right) =\left( y,mx\right) $. Passing to the
realization inside $V\otimes V^{\ast }$ we obtain a geometric realization of 
$\Phi ^{-1}\left( \Gamma\left( m\circ R_{w_{0}}\right) \right) $,
also denoted by $\Gamma\left( m\circ R_{w_{0}}\right) $:

\begin{proposition}
	$\Gamma\left( m\circ R_{w_{0}}\right) =\{v\otimes \rho _{\mu }^{\ast
	}\left( m\right) \varepsilon \in G\cdot \left( v_{0}\otimes \varepsilon
	_{0}\right) :\ker \varepsilon =v^{\bot }\}.$
\end{proposition}

%-- follow  the proposition describing graph $R_{w_0}$.

Now we  have the setup to prove that $f_{H}$ is real on $\Gamma%
\left( m\circ R_{w_{0}}\right) $.
%Isso foi provado para $\mathbb{P}^{n}$ e 
This is essential to obtain real  Lagrangian thimbles. With the realization of 
$G/Z_{\mu }$ as an orbit in $V\otimes V^{\ast }$ the proof that
$f_{H}$ is real greatly simplifies. Actually, this is not only true for elements
$m\in T$, but for more general linear transformations of 
$V$ (or more precisely, of $V^{\ast }$). 

Before stating the result, observe that the function $f_{H}$
is a priori defined on the orbit $G\cdot \left( v_{0}\otimes
\varepsilon _{0}\right) $ and is given by $f_{H}\left( v\otimes
\varepsilon \right) =\varepsilon \left( \rho _{\mu }\left( H\right) v\right) 
$. From this expression we see that  $f_{H}$ extends to a linear functional of
$V\otimes V^{\ast }$, that is, it is defined  on points 
outside  the orbit $G\cdot \left( v_{0}\otimes \varepsilon
_{0}\right) $ as well. 

\begin{proposition}
	Let  $D:V\rightarrow V$ be a linear transformation that is diagonalizable  
	on a basis adapted to the root subspaces and consider the set
	\begin{equation*}
	D_{2}\left( \Gamma\left( R_{w_{0}}\right) \right) =\{v\otimes D^{\ast
	}\varepsilon \in V\otimes V^{\ast }:\ker \varepsilon =v^{\bot }\}
	\end{equation*}%
	where $D^{\ast }\varepsilon =\varepsilon \circ D$.  Suppose that $D$ has 
	real eigenvalues. Then, $f_{H}$ assumes real values on  $D_{2}\left( 
	\Gamma\left( R_{w_{0}}\right) \right) $. 
\end{proposition}

\begin{proof}
	If $v\otimes D\varepsilon \in D_{2}\left( \Gamma\left(
	R_{w_{0}}\right) \right) $ then
	\begin{equation*}
	f_{H}\left( v\otimes D^{\ast }\varepsilon \right) =\varepsilon \left( D\rho
	_{\mu }\left( H\right) v\right) =\mathrm{tr}\left( \left( v\otimes
	\varepsilon \right) D\rho _{\mu }\left( H\right) \right) .
	\end{equation*}%
	On a basis adapted to the root subspaces, $D\rho _{\mu }\left(
	H\right) $ is diagonal with real eigenvalues. If this basis is 
	orthonormal, then $v\otimes \varepsilon $ has a Hermitian matrix, and
	therefore real  diagonal entries. Hence, the last term of the equality above 
	is real, and  $f_{H}\left( v\otimes D^{\ast }\varepsilon \right) $ is real as well. 
\end{proof}

\begin{corollary}
	If $m\in T$ satisfies $m^{2}=1$ then  $f_{H}$ is real on $\Gamma%
	\left( m\circ R_{w_{0}}\right) $.
\end{corollary}

\begin{proof}
	In fact, if $m^{2}=1$ then the eigenvalues of $m$ are $\pm 1$ and
	since $m\in T$, $\rho _{\mu }\left( m\right) =\pm \mathrm{id}$  
	on the root spaces.   
\end{proof}

Further properties of Lagrangian submanifolds  inside products of flags and their intersection numbers
are described in \cite{GSMV}. 

\section{Minimal semisimple orbits}\label{MSO}
%$\mathbb{P}^{2n}$}

We now focus our attention on the case of minimal semisimple orbits, by
considering  the orbits of  $\mathfrak{sl}( n+1,\mathbb{C}) $ of smallest dimension. The
corresponding flag manifolds are 
%the $\mathrm{SU}(n+1)$ orbits  of minimal dimension:
  $\mathbb{F}_{H_{0}}=\mathbb{P}^{n}$ and $%
\mathbb{F}_{H_{0}^{\ast }}=\mathrm{Gr}_{n}( n+1)=  \mathbb{P}^{n^*}  $ 
for
$
H_{0}={\tiny \left( 
\begin{array}{cc}
n & 0 \\ 
0 & -1_{ n \times  n }%
\end{array}%
\right)} $
$ H_{0}^{\ast }={\tiny \left( 
\begin{array}{cc}
1_{n \times n } & 0 \\ 
0 & -n%
\end{array}%
\right) }.
$
%and are $%
%\mathrm{SU}\left( n+1\right)$-adjoint orbits, in the space of 
%Hermitian matrices of 
These are dual to each other, so it suffices to consider the case of $\mathbb{P}^{n}$.%

For  minimal flags 
%of  $\mathfrak{sl}\left( n+1,\mathbb{C}\right) $ 
it is possible 
to describe  real  Lagrangian thimbles of  $f_{H}$ for all 
singularities via the graphs  $\mathrm{graph}( m\circ
R_{w_{0}}) $ with $m\in T$. This happens because for each of these 
singularities there are elements  $m\in T$ such that  $\mathrm{Hess}\left( f_{H}\right) $
restricted to  $\Gamma( m\circ R_{w_0}) $ is either
positive definite or negative definite. Together with the fact (proved
below) that the imaginary part of  $f_{H}$ is constant over the corresponding graphs,
we obtain the stable and unstable manifolds of 
$\grad \left( \text{Re}f_{H}\right) $ and consequently also real
Lagrangian thimbles.

Our general construction specializes to this situation as follows:

\begin{itemize}
 \item  The diagonal action of  $\mathrm{Sl}( n+1,\mathbb{C}) $
	on  $\mathbb{F}_{H_{0}}\times \mathbb{F}_{H_{0}^{\ast }}=
	\mathbb{P}^{n}\times  \mathbb{P}^{n^*}$ has 2 orbits. An open dense one
	formed by pairs of transversal vectors $\left( V,W\right)
	\in \mathbb{P}^{n}\times  \mathbb{P}^{n^*} $ with $V\cap
	W=\{0\}$; and another orbit formed by  vectors $\left( V,W\right) \in 
	\mathbb{P}^{n}\times \mathbb{P}^{n^*}  $ with $V\subset W$.
		
	\item The diffeomorphism between the  open orbit and the adjoint orbit  
	$\mathcal O(H_{0}) $ associates to a pair  
	$\left( V,W\right) \in \mathbb{P}^{n}\times  \mathbb{P}^{n^*}  $ with  $V\cap W=\{0\}$ the
	linear transformation $T\colon \mathbb{C}^{n+1}\rightarrow \mathbb{C}^{n+1}$ with $Tv=nv$ if $v\in V$
	and $Tv=-v$ if $v\in W$.
	
	\item The map $R_{w_{0}}\colon \mathbb{P}^{n}\rightarrow   \mathbb{P}^{n^*} $
	 associates to a subspace of dimension $1$ of  $%
	\mathbb{C}^{n+1}$ its orthogonal complement with respect to the canonical Hermitian form
	%$\left( \cdot ,\cdot \right) $ 
	of $\mathbb{C}^{n+1}$.
	
	\item $\mathcal{W}$ is the group of permutation of $n+1$ elements.
	
	\item The set of critical points of the potential in  $\mathbb{P}^{n}$ (orbit of the 
	Weyl group at the origin) has $n+1$ elements which are the subspaces  $%
	[e_{j}]$, \ $j=1,\ldots ,n+1$, generated by the vectors  of the canonical basis of $%
	\mathbb{C}^{n+1}$. The origin is given by  $[e_{1}]$, so that under the identifications
	this origin gets identified to  $H_{0}$; whereas  $[e_{j}]$ gets
	identified to $wH_{0}$ for any permutation  $w$ such that $w\left(
	1\right) =j$.
	
	\item The roots are $\alpha _{ij}$ with $i\neq j$, with corresponding eigenspaces
	 generated by the elementary matrices $X_{\alpha
		_{ij}}=X_{ij}$ (with $1$ in the position $ij$ and $0$ elsewhere).
	
	\item The roots $\alpha $ with $\alpha \left( H_{0}\right) <0$ are 
	given by $\alpha _{j1}$ with $2\leq j\leq n+1$ and consequently the 
	tangent space at the origin is identified with the space of column matrices 
	${\tiny
	\left( 
	\begin{array}{cc}
	0 & 0 \\ 
	\ast & 0_{n \times n }%
	\end{array}%
	\right) }$.

	\item The tangent spaces at the other  critical points are obtained via
	permutation: let $w$ be the permutation such that $w\left( 1\right) =j 
	$, then $\alpha( wH_{0}) <0$ if and only if $\alpha
	=\alpha _{ij}$ with $i\neq j$. Thus, the tangent space at  $[e_{j}]\approx
	wH_{0}$ is formed by matrices whose nonzero entries belong to  the   $j$-th column and
	that have a zero on entry $jj$. 
\end{itemize}

Assume now, once and for all that $n$ is even, hence we are working in $\mathfrak{sl}(n+1)$ with
matrices that have an odd number of diagonal entries. 

\begin{definition}
	\label{defememaismenos} Given a singularity $[e_{j}]$ define $m_{j}^{\pm
	}\in T=\exp \left( i\mathfrak{h}_{\mathbb{R}}\right) $ as follows:
	$$m_j^\pm= \mp (-1)^{j} \mathrm{diag}(1,\dots, 1, \pm 1_j, -1, \dots, -1)$$ 
	(the subindex indicates position $j$).

	%\begin{enumerate}
	%\item $m_{j}^{+}=\mathrm{diag}\{-1,\ldots ,-1_{j},1,\ldots ,1\}$ (the subindex indicates position $j$) e
	%
	%\item $m_{j}^{-}=\mathrm{diag}\{1,\ldots ,1_{j-1},-1,\ldots ,-1\}$.
	%\end{enumerate}
	%
	%\item If $j$ is odd, then
	%
	%\begin{enumerate}
	%\item $m_{j}^{+}=\mathrm{diag}\{1,\ldots ,1_{j},-1,\ldots ,-1\}$ and
	%
	%\item $m_{j}^{-}=\mathrm{diag}\{-1,\ldots ,-1_{j-1},1,\ldots ,1\}$.
	%\end{enumerate}
	%\end{enumerate}
\end{definition}

The fact that the number of diagonal entries is odd, guaranties that in all cases $\det
m_{j}^{\pm }=1$ and, therefore, $m_{j}^{\pm }$ indeed belongs to  $T$. (Although
in the first and last cases this is also true for even $n$.)

\begin{proposition}
	Consider the singularity $\left[ e_{j}\right] \approx wH_{0}$. The restriction of 
	$\mathrm{Hess}( f_{H}) $ to the  tangent space $%
	T_{\left[ e_{j}\right] } \Gamma( m\circ R_{w_{0}}) $
	 is positive definite if  $m=m_{j}^{+}$ and negative definite if $%
	m=m_{j}^{-}$.
\end{proposition}

\begin{proof}
	By proposition \ref{prophess} the restriction of $\mathrm{Hess}( f_{H}) $ 
	is diagonal in the basis given by roots $\alpha
	\left( wH_{0}\right) <0$. The diagonal entries associated to the 2
	dimensional subspace corresponding to the root
	$\alpha $ are:
	\begin{equation*}
	-2\alpha \left( wH_{0}\right) \alpha \left( H\right) e^{-i\alpha \left(
		H_{1}\right) }\quad \text{and} \quad \alpha \left( wH_{0}\right) <0
	\end{equation*}%
	where $H_{1}$ is such that $m=\exp iH_{1}$.
	
	As mentioned earlier, if $wH_{0}\approx \left[ e_{j}\right] $ then
	the roots  $\alpha $ such that $\alpha \left( wH_{0}\right) <0$ are given by
	$\alpha _{kj}$ with $k\neq j$. Also, recall that from the start  $H$ was taken in the positive Weyl chamber $
		\mathfrak{h}_{\mathbb{R}}^{+}$.  Therefore, for these roots we have
	
	\begin{itemize}
		\item $\alpha _{kj}\left( H\right) >0$ if $k<j$, since $\alpha _{kj}>0$, and
		\item $\alpha _{kj}\left( H\right) <0$ if $k>j$, since $\alpha _{kj}<0$.
	\end{itemize}
Now, if $m=\exp iH_{1}=\mathrm{diag}\{\varepsilon _{1},\ldots ,\varepsilon
	_{n}\}$, then $e^{-i\alpha _{kj}\left( H_{1}\right) }=\varepsilon
	_{k}\varepsilon _{j}$. Hence,
	
	\begin{itemize}
		\item for $m_{j}^{+}=\exp iH_{1}$, we have
		\begin{equation*}
		e^{-i\alpha _{kj}\left( H_{1}\right) }=\left\{ 
		\begin{array}{ccc}
		1 & \mathrm{if} & k<j \\ 
		-1 & \mathrm{if} & k>j%
		\end{array}%
		\right. \quad \text{and}
		\end{equation*} 
\item for $m_{j}^{-}=\exp iH_{1}$, we have
		\begin{equation*}
		e^{-i\alpha _{kj}\left( H_{1}\right) }=\left\{ 
		\begin{array}{ccc}
		-1 & \mathrm{if} & k<j \\ 
		1 & \mathrm{if} & k>j.%
		\end{array}%
		\right.
		\end{equation*}
	\end{itemize}
	Combining the signs of  $e^{-i\alpha _{kj}\left( H_{1}\right) }$ and of $\alpha
	_{kj}\left( H\right) $ we see that the Hessian $\mathrm{Hess}\left( f_{H}\right) $ 
	is positive definite on $T_{\left[ e_{j}\right] }\Gamma
	( m_{j}^{+}\circ R_{w_{0}})  $ and  negative definite on $T_{%\mathbb{P}^{n}
		\left[ e_{j}\right] }\Gamma( m_{j}^{-}\circ
	R_{w_{0}})  $.
\end{proof}

The goal now is to show that the imaginary part of  $f_{H}$ is 
constant on  $\Gamma( m_{j}^{\pm }\circ
R_{w_{0}}) $. These graphs intercept the zero section  $%
\Gamma( R_{w_{0}}) $ where $f_{H}$ is real. So,
we wish to show that $f_{H}$ is real on $\Gamma(
m_{j}^{\pm }\circ R_{w_{0}}) $.

For a transversal pair $\left( V,W\right) \in \mathbb{P}^{n}\times   \mathbb{P}^{n^*}$ denote by $\Phi \left( V,W\right) $ the
linear transformation in  $\mathcal{O}\left( H_{0}\right) $
corresponding to the pair. As mentioned earlier, $\Phi \left( V,W\right)
v=nv$ if $v\in V$ and $\Phi \left( V,W\right) w=-w$ then $w\in W$.
To show that $f_{H}$ is real on the graph $\Gamma ( m_{j}^{\pm
}\circ R_{w_{0}}) $ we will prove that the diagonal of $\Phi \left(
V,W\right) $ has real entries.

The following calculations work for any $m\in T$ such that $m^{2}=1$. So,
fix once and for all $m\in T$ such that $m^{2}=1$.
Take $\left[ u\right] \in \mathbb{P}^{n}$. Then $R_{w_{0}}\left[ u%
\right] =\left[ u\right] ^{\bot }$ and $\Phi \left( [ u] ,\left[ u%
\right] ^{\bot }\right) $ is a Hermitian matrix whose diagonal entries are
real. Since, $m\circ R_{w_{0}}\left[ u\right] =m\left[ u\right] ^{\bot }$,
we ought to show that $\Phi \left( \left[ u\right] ,m\left[ u\right]
^{\bot }\right) $ has real diagonal entries. Here assume that  $%
\left[ u\right] $ and $m\left[ u\right] ^{\bot }$ are transversal, that is
$\left( u,mu\right) \neq 0$.

Suppose $|u|=1$ and take an orthonormal basis $\{v_{1},\ldots ,v_n\}$ of $%
\left[ u\right] ^{\bot }$ with respect to the canonical  Hermitian form
$\left( \cdot ,\cdot \right) $ of $\mathbb{C}^{n+1}$. The basis $%
\{u,v_{1},\ldots ,v_n\}$ is orthonormal in $\mathbb{C}^{n+1}$ and since $%
m\in \mathrm{SU}\left( n\right) $ the basis $\beta =\{mu,mv_{1},\ldots
,mv_n\}$  is orthonormal as well, whereas the basis $\gamma
=\{u,mv_{1},\ldots ,mv_n\}$ is not orthonormal. By the definition of 
$\Phi $ the matrix of  $\Phi \left( \left[ u\right] ,m\left[ u\right]
^{\bot }\right) $ on the basis $\gamma $ is given by 
\begin{equation*}
\left[ \Phi \left( \left[ u\right] ,m\left[ u\right] ^{\bot }\right) \right]
_{\gamma }=\mathrm{diag}\{n,-1,\ldots ,-1\}.
\end{equation*}%
The matrices for the change of basis between $\beta $ and $\gamma $ are
\begin{equation*}
\left[ I\right] _{\beta }^{\gamma }=\left( 
\begin{array}{cccc}
\left( u,mu\right) & 0 & \cdots & 0 \\ 
\left( u,mv_{1}\right) & 1 & \ddots & \vdots \\ 
\vdots & \vdots & \ddots & 0 \\ 
\left( u,mv_n\right) & 0 & \cdots & 1%
\end{array}%
\right)
\end{equation*}
with inverse 
\begin{equation*}
\left[ I\right] _{\gamma }^{\beta }=\left( 
\begin{array}{cccc}
1/\left( u,mu\right) & 0 & \cdots & 0 \\ 
-\left( u,mv_{1}\right) /\left( u,mu\right) & 1 & \ddots & \vdots \\ 
\vdots & \vdots & \ddots & 0 \\ 
-\left( u,mv_n\right) /\left( u,mu\right) & 0 & \cdots & 1%
\end{array}%
\right) .
\end{equation*}%

Therefore, $\left[ \Phi \left( \left[ u\right] ,m\left[ u\right] ^{\bot
}\right) \right] _{\beta }=\left[ I\right] _{\beta }^{\gamma }\left[ \Phi
\left( \left[ u\right] ,m\left[ u\right] ^{\bot }\right) \right] _{\gamma }%
\left[ I\right] _{\gamma }^{\beta }$ is given by 
\begin{equation*}
\left[ \Phi \left( \left[ u\right] ,m\left[ u\right] ^{\bot }\right) \right]
_{\beta }=\left( 
\begin{array}{cccc}
n & 0 & \cdots & 0 \\ 
\nu \left( u,mv_{1}\right) & -1 & \ddots & 
\vdots \\ 
\vdots & \vdots & \ddots & 0 \\ 
\nu \left( u,mv_n\right) & 0 & \cdots
& -1%
\end{array}%
\right) ,
\end{equation*}
 where we have set $\nu\ce (n+1)/\left( u,mu\right)$.

We now claim that the diagonal elements  of $\Phi \left( \left[ u%
\right] ,m\left[ u\right] ^{\bot }\right) $ are  given by $\left( \Phi
\left( \left[ u\right] ,m\left[ u\right] ^{\bot }\right) e_{j},e_{j}\right) $
where $e_{j}$ is an element of the canonical basis. To see this,
 take an arbitrary $x\in 
\mathbb{C}^{n+1}$ and write $\left[ \Phi \left( \left[ u\right]
,m\left[ u\right] ^{\bot }\right) \right] _{\beta }=A+B$ with 
\begin{equation*}
A=\left( 
\begin{array}{cccc}
n & 0 & \cdots & 0 \\ 
\nu \left( u,mv_{1}\right)  & 0 & \ddots & 
\vdots \\ 
\vdots & \vdots & \ddots & 0 \\ 
\nu \left( u,mv_n\right)  & 0 & \cdots
& 0%
\end{array}%
\right)
\end{equation*}%
and
\begin{equation*}
B=\left( 
\begin{array}{cccc}
0 & 0 & \cdots & 0 \\ 
0 & -1 & \ddots & \vdots \\ 
\vdots & \vdots & \ddots & 0 \\ 
0 & 0 & \cdots & -1%
\end{array}%
\right) .
\end{equation*}%
In coordinates 
%of  $x$ on the basis $\beta $ are given by 
$x=\left(
x,mu\right) mu+\left( x,mv_{1}\right) mv_{1}+\cdots +\left(
x,mv_n\right) mv_n$. So, 
$$
\left[ Ax\right] _{\beta }=\left( 
\begin{array}{cccc}
n & 0 & \cdots & 0 \\ 
\nu \left( u,mv_{1}\right)& 0 & \ddots & 
\vdots \\ 
\vdots & \vdots & \ddots & 0 \\ 
\nu \left( u,mv_n\right)
 & 0 & \cdots
& 0%
\end{array}%
\right) \left( 
\begin{array}{c}
\left( x,mu\right) \\ 
\left( x,mv_{1}\right) \\ 
\vdots \\ 
\left( x,mv_n\right)%
\end{array}%
\right) =
%\end{equation*}
%\begin{equation*}
%\quad \quad \quad\quad 
\left( 
\begin{array}{c}
n\left( x,mu\right) \\ 
\nu \left( u,mv_{1}\right) \left( x,mu\right)  \\ 
\vdots \\ 
\nu \left( u,mv_n\right) \left( x,mu\right) 
\end{array}%
\right)
$$
and%
\begin{equation*}
\left[ Bx\right] _{\beta }=\left( 
\begin{array}{cccc}
0 & 0 & \cdots & 0 \\ 
0 & -1 & \ddots & \vdots \\ 
\vdots & \vdots & \ddots & 0 \\ 
0 & 0 & \cdots & -1%
\end{array}%
\right) \left( 
\begin{array}{c}
\left( x,mu\right) \\ 
\left( x,mv_{1}\right) \\ 
\vdots \\ 
\left( x,mv_n\right)%
\end{array}%
\right) =\left( 
\begin{array}{c}
0 \\ 
-\left( x,mv_{1}\right) \\ 
\vdots \\ 
-\left( x,mv_n\right)%
\end{array}%
\right) .
\end{equation*}%
Since  $\beta $ is an orthonormal basis, we have that $\left( \Phi \left( \left[ u\right] ,m%
\left[ u\right] ^{\bot }\right) x,x\right) $ is given by the sum of

\begin{itemize}
	\item $\left( Ax,x\right) =n|\left( x,mu\right) |^{2}+\frac{
		n }{\left( u,mu\right) }\left( x,mu\right) \sum_{j=1}^{n}\left(
	u,mv_{j}\right) \overline{\left( x,mv_{j}\right) }$ and
	
	\item $\left( Bx,x\right) =-\sum_{j=1}^{n}|\left( x,mv_{j}\right) |^{2}$.
\end{itemize}
In this sum, the only part that is not evidently real is 
\begin{equation*}
\frac{n}{\left( u,mu\right) }\left( x,mu\right)
\sum_{j=1}^{n}\left( u,mv_{j}\right) \overline{\left( x,mv_{j}\right) }.
\end{equation*}%
To analyze this part, start up observing that $\left(
u,mu\right) \in \mathbb{R}$ since $m$ is an isometry of $\left( \cdot
,\cdot \right) $. Thus, $\left( u,mu\right) =\left( mu,u\right) =%
\overline{\left( u,mu\right) }$. So, we ought to verify that
\begin{equation}
\left( x,mu\right) \sum_{j=1}^{n}\left( u,mv_{j}\right) \overline{\left(
	x,mv_{j}\right) }\in \mathbb{R}  \label{forsomdevesereal}
\end{equation}%
whenever $x$ is an element of the canonical basis. This works specifically for an element $x$
of the canonical basis, because in this case  $mx=\pm x$, that is, $x$
belongs to an eigenspace of  $m$.
The sum in  (\ref{forsomdevesereal}) can be rewritten as $%
\sum_{j=1}^{n}\left( mu,v_{j}\right) \overline{\left( mx,v_{j}\right) }$,
because  $m\in \mathrm{SU}( n) $.  
 It can also be expressed as 
\begin{equation*}
\sum_{j=1}^{n}\left( mu,\left( mx,v_{j}\right) v_{j}\right) =\left(
mu,\sum_{j=1}^{n}\left( mx,v_{j}\right) v_{j}\right) .
\end{equation*}%
The second factor is the orthogonal projection $\mathrm{proj}\left(
mx\right) $ of  $mx$ over $\left[ v_{1},\ldots ,v_n\right] =\left[ u%
\right] ^{\bot }$. Since $\left( mu,\mathrm{proj}\left( mx\right) \right)
=\left( \mathrm{proj}\left( mu\right) ,mx\right) $ it then follows that the sum in  (\ref%
{forsomdevesereal}) is
\begin{equation*}
\left( mx,u\right) \left( \sum_{j=1}^{n}\left( mu,v_{j}\right)
v_{j},mx\right) .
\end{equation*}%
Finally, since $mx=\pm x$ (as it happens for the elements of the canonical basis) the previous expression becomes
\begin{equation}
\left( x,u\right) \left( \sum_{j=1}^{n}\left( mu,v_{j}\right)
v_{j},x\right) .  \label{forsomdevesereal1}
\end{equation}%
We can now prove this expression is  real.

\begin{lemma}
	Expression  (\ref{forsomdevesereal1}) is real.
\end{lemma}

\begin{proof}
	Let $E_{\pm }$ be the eigenspaces associated to the eigenvalues $\pm 1$ of $%
	m$ (since $m^{2}=1$), and write  $u=u^{+}+u^{-}\in E_{+}\oplus E_{-}$ (that is,
	$u^{+}=1/2\left( u+mu\right) $ and $u^{+}=1/2\left( u-mu\right) $. Then,
	for each index $j$, $0=\left( u,v_{j}\right) =\left(
	u^{+},v_{j}\right) +\left( u^{-},v_{j}\right) $, that is, 
	\begin{equation*}
	\left( u^{-},v_{j}\right) =-\left( u^{+},v_{j}\right) .
	\end{equation*}%
	It follows that $\left( mu,v_{j}\right) =\left( u^{+}-u^{-},v_{j}\right) =2\left(
	u^{+},v_{j}\right) =-2\left( u^{-},v_{j}\right) $. Suppose, for example, that 
	$x\in E_{+}$. Then, $\left( x,u\right) =\left( x,u^{+}\right) $, since $%
	E_{+}$ is orthogonal to  $E_{-}$, given that $m$ is unitary. Thus
	(\ref{forsomdevesereal1}) can be rewritten as
	\begin{eqnarray*}
		\sum_{j=1}^{n}\left( x,u\right) \left( mu,v_{j}\right) \left(
		v_{j},x\right) &=&2\sum_{j=1}^{n}\left( x,u^{+}\right) \left(
		u^{+},v_{j}\right) \left( v_{j},x\right) \\
		&=&2\sum_{j=1}^{n}\left( x,u^{+}\right) \left( u^{+},\left( x,v_{j}\right)
		v_{j}\right) \\
		&=&2\left( x,u^{+}\right) \sum_{j=1}^{n}\left( u^{+},\left( x,v_{j}\right)
		v_{j}\right) \\
		&=&2\left( x,u^{+}\right) \left( u^{+},\sum_{j=1}^{n}\left( x,v_{j}\right)
		v_{j}\right) .
	\end{eqnarray*}%
	The sum inside the Hermitian form is
	\begin{equation*}
	\sum_{j=1}^{n}\left( x,v_{j}\right) v_{j}=\left( x,u\right)
	u+\sum_{j=1}^{n}\left( x,v_{j}\right) v_{j}-\left( x,u\right) u=x-\left(
	x,u\right) u
	\end{equation*}%
	since  $\{u,v_{1},\ldots ,v_n\}$ is an orthonormal basis. So, the last
	term is given by 
	$$
	2\left( x,u^{+}\right) \left( u^{+},x\right) -2\left( x,u^{+}\right) \left(
	u^{+},\left( x,u\right) u\right) =$$
	$$2\left( x,u^{+}\right) \left(
	u^{+},x\right) -2\left( x,u^{+}\right) \left( u,x\right) \left(
	u^{+},u\right) .
	$$
	To see that this is real observe that the first term of the right hand side is $|\left(
	x,u^{+}\right) |^{2}$. As for the second term, $\left( u,x\right) =\left(
	u^{+},x\right) $ is $\left( u^{+},u\right) =\left( u^{+},u^{+}\right) $,
	so, the second term  is  $2|\left( x,u^{+}\right) |^{2}\left(
	u^{+},u^{+}\right) $ which is also real.
\end{proof}

Summing up, we have obtained:

\begin{proposition}
	If  $m\in T$ with $m^{2}=1$ and \ $V\in \mathbb{P}^{n}$ is such that
	$V$ does not belong to  $m\circ R_{w_{0}}\left( V\right) =mV^{\bot }$ then
	the matrix of  $\Phi \left( V,m\circ R_{w_{0}}V\right) $ 
	has real diagonal entries (in the canonical basis).
\end{proposition}

In conclusion:

\begin{theorem}\label{Morse}
	Let $m_{j}^{\pm }$ be as in definition \ref{defememaismenos}. Then,
	the stable and unstable manifolds of  $\grad \left( 
	\text{Re}f_{H}\right) $ at the critical point  $\left[ e_{j}\right] $ are
	open in  the graph $\Gamma( m_{j}^{\pm }\circ R_{w_{0}}) $. The 
	real  Lagrangian thimbles are closed balls contained in  the graph $\Gamma
	( m_{j}^{\pm }\circ R_{w_{0}}) $.
\end{theorem}

\section{Acknowledgements}  

Gasparim was partially supported by a Simons Associateship Grant of the 
Abdus Salam International Centre for Theoretical Physics and by the 
Vice-Rector\'\i a de Investigaci\'on y Desarrollo Tecnol\'ogico 
of Universidad Cat\'olica del Norte, Chile.
San Martin was partially supported by CNPq grant no.\ 303755/09-1 and Fapesp grant no.\ 2018/13481.
We thank F. Valencia for pointing out a few corrections  needed on an early version of the text.

\end{document}